\documentclass[11pt, a4paper]{article}
\usepackage[margin=2.6cm]{geometry}
\usepackage{amsthm}
\usepackage{amsmath}
\usepackage{amssymb} 
\usepackage{mathrsfs} 
\usepackage{amsfonts} 
\usepackage{esint} 
\usepackage{array} 
\usepackage{latexsym}
\usepackage{lmodern} 
\usepackage{graphicx} 
\usepackage{epsfig} 
\usepackage[all]{xy}
\usepackage[new]{old-arrows}
\usepackage{textcomp}
\usepackage{stmaryrd}
\DeclareGraphicsExtensions{.eps,.pdf,.jpeg,.png} 
\usepackage{fancyhdr}
\usepackage{multirow} 
\usepackage{comment}
\usepackage{pgf,tikz}
\usepackage{paralist}
\usepackage{faktor}
\usepackage{dirtytalk}
\usepackage{blkarray,bigstrut}
\usepackage{nicematrix}
\usepackage{enumerate}
\usepackage[most]{tcolorbox}
\usepackage{titlesec}
\usepackage[left]{lineno}

\makeatletter
\renewcommand*\env@matrix[1][*\c@MaxMatrixCols c]{%
  \hskip -\arraycolsep
  \let\@ifnextchar\new@ifnextchar
  \array{#1}}
\makeatother

\newtheorem{corollary}{Corollary}[section]
\newtheorem{proposition}{Proposition}[section]
\newtheorem{theorem}{Theorem}[section]
\newtheorem{lemma}{Lemma}[section]

\theoremstyle{definition}
\newtheorem{definition}{Definition}[section]
\newtheorem{remark}{Remark}[section]
\newtheorem{example}{Example}[section]

\usepackage{etoolbox} 
\preto\theorem{\vspace{3pt}}
\apptocmd\endtheorem{\vspace{3pt}}{}{}

\preto\definition{\vspace{3pt}}
\apptocmd\enddefinition{\vspace{3pt}}{}{}

\preto\proposition{\vspace{3pt}}
\apptocmd\endproposition{\vspace{3pt}}{}{}

\preto\corollary{\vspace{3pt}}
\apptocmd\endcorollary{\vspace{3pt}}{}{}

\preto\lemma{\vspace{3pt}}
\apptocmd\endlemma{\vspace{3pt}}{}{}

\preto\remark{\vspace{3pt}}
\apptocmd\endremark{\vspace{3pt}}{}{}

\makeatletter

\titleformat{\section}
  {\centering\normalfont\Large\bfseries}
  {\thesection}{1em}{}

\titleformat{\subsubsection}[runin]
  {\normalfont\normalsize\bfseries}   
  {\thesubsubsection}                 
  {1em}                               
  {}                                  
  [. ]                                

\titlespacing*{\subsubsection}{0pt}{2.0ex plus .2ex}{0.8em}

\newcommand{\spacearound}[3]{
  \preto{#1}{\vspace{#2}}
  \apptocmd{\csname end#1\endcsname}{\vspace{#3}}{}{}
}
\makeatother

\date{}

\title{Orientation of alcoves in affine Weyl groups}

\author{Nathan Chapelier-Laget}

\usepackage[dvipsnames]{xcolor}
\usepackage[backref = page]{hyperref}
\hypersetup{
colorlinks =  true, 
linkcolor = purple!80, 
citecolor =teal
}

\begin{document}


\maketitle

\begin{abstract}
Let $W$ be an irreducible Weyl group and $W_a$ its affine Weyl group. In \cite{NC1} the author introduced an affine variety $\widehat{X}_{W_a}$, called the Shi variety of $W_a$, whose integral points are in bijection with $W_a$. The set of irreducible components of $\widehat{X}_{W_a}$ provided results at the intersection of group theory, combinatorics and geometry. In this article we express the notion of orientation of alcoves in terms of the first cohomology group of $W$ and the irreducible components of the Shi variety. We also provide modular equations in terms of Shi coefficients that describe explicitly the property of having the same orientation.
\end{abstract}

\tableofcontents

\vspace{1cm}

\noindent \href{https://www.nathanchapelier.fr/}{\textcolor{gray}{Nathan Chapelier-Laget}}  \\
nathan.chapelier@univ-littoral.fr \\
Université du Littoral Côte d'Opale, \\
Calais, 62228, France

\clearpage

\newpage

\section{Introduction}

\subsection{General definitions}
Let $V$ be a Euclidean space with inner product $\langle -, -\rangle$. Let $\Phi$ be an irreducible crystallographic root system in $V$ with simple system $\Delta$. We set $\Delta = \{\alpha_1, \dots, \alpha_n \}$ and $\Phi^+ =\{\beta_1,\dots, \beta_m\}$ with $n=|\Delta|$ and $m=|\Phi^+|$. We assume here that $\Phi$ is essential, i.e.  $\mathbb{Z}\Phi \otimes_{\mathbb{Z}}\mathbb{R}= V$. In this article, when we say \say{root system} it always means an irreducible, crystallographic, essential root system.

 For $\alpha \in \Phi$ we write $\alpha^{\vee}:= \frac{2\alpha}{\langle \alpha, \alpha \rangle}$. The dual root system of $\Phi$ is defined by $\Phi^{\vee} := \{\alpha^{\vee}~|~\alpha \in \Phi\}$, with dual simple system $\Delta^{\vee} := \{\alpha^{\vee}~|~\alpha \in \Delta\}$.

\subsubsection{Height}
Let $\alpha \in \Phi$ such that  $\alpha = a_1\alpha_1 + \cdots + a_n\alpha_n$ with $a_i \in \mathbb{Z}$. The height of $\alpha$ (with respect to $\Delta$) is defined by the number $\mathsf{h}(\alpha) = a_1 + \cdots+ a_n$. 
 The height of the dual root $\alpha^{\vee} = b_1\alpha_1^{\vee} + \cdots + b_n\alpha_n^{\vee}$ (with respect to $\Delta^{\vee}$) is defined by $\mathsf{h}(\alpha^{\vee}) = b_1 + \cdots + b_n$.
 
  We also have $\mathsf{h}(x + y) = \mathsf{h}(x) + \mathsf{h}(y)$ and $\mathsf{h}(-x)=-\mathsf{h}(x)$ for all $x, y, x+y \in \Phi$ (respectively for all $x, y, x+y \in \Phi^{\vee}$).

From the classification of irreducible crystallographic root systems there are at most two possible root lengths in $\Phi$. A root is called short if it has minimal length. We denote by $\theta$ the \emph{highest short root} of $\Phi$.

 \subsubsection{Index of connection} 
 We set once and for all $\Psi := \Phi^{\vee}$ and  $\varepsilon_i := \alpha_i^{\vee}$ for any $i=1,\dots,n$.
Let $\mathbb{Z}\Psi $ be the coroot lattice and let us write $\mathbb{Z}\Psi  = \mathbb{Z}\varepsilon_1 \oplus \cdots\oplus \mathbb{Z}\varepsilon_n$. We define its dual lattice $\mathbb{Z}\Psi^*$ as $
\mathbb{Z}\Psi^{*} := \{ x \in V ~|~\langle x, y \rangle \in \mathbb{Z} ~\forall y \in \mathbb{Z}\Psi\}.
$
The lattice $\mathbb{Z}\Psi^{*}$ is called the \emph{weight lattice}. The weight lattice can be decomposed as 
\begin{equation}
\mathbb{Z}\Psi^{*}= \mathbb{Z}\omega_1\oplus \cdots\oplus \mathbb{Z}\omega_n
\end{equation}
where $\omega_i$ is such that $\langle \varepsilon_i, \,\omega_j \rangle = \delta_{ij}$. The elements $\omega_i$ are called the \emph{fundamental weights}.

The lattice $\mathbb{Z}\Phi$ is a sublattice of $\mathbb{Z}\Psi^{*}$. The \emph{index of connection} of $\Phi$ is 
the cardinality of the following quotient group, and we denote it by
\begin{equation}\label{def index of connection}
f_{\Phi} := \left | \faktor{\mathbb{Z}\Psi^{*}}{\mathbb{Z}\Phi} \right |.
\end{equation}

\subsubsection{Weyl group and affine Weyl group}
Let $W$ be the \emph{Weyl group} associated to $\mathbb{Z}\Phi$, that is the maximal (for inclusion) reflection subgroup of $O(V)$ admitting $\mathbb{Z}\Phi$ as a $W$-equivariant lattice.  
We identify $\mathbb{Z}\Phi$ and the group of its associated translations and we denote by $\tau_x$ the translation corresponding to $x \in \mathbb{Z}\Phi$.

Let $k \in \mathbb{Z}$, $\alpha \in \Phi$. We define the affine reflection $s_{\alpha,k} \in \text{Aff}(V)$ by 
\begin{equation}\label{affine reflection}
s_{\alpha,k}(x)=x-\left(\langle \alpha^{\vee}, x \rangle-k\right)\alpha.
\end{equation}
The group generated by all the affine reflections $s_{\alpha,k}$ with $\alpha \in \Phi$ and $k \in \mathbb{Z}$ is called the \emph{affine Weyl group} associated to $\Phi$. We denote it by
\begin{equation}
    W_a := \langle s_{\alpha,k}~|~\alpha \in \Phi, ~k \in \mathbb{Z} \rangle.
\end{equation}

The group $W_a$ admits a Coxeter group structure, with set of Coxeter generators $S_a := \{s_{\alpha_1},\dots,s_{\alpha_n}\} \cup \{s_{\theta,1}\}$. For short, set $S_a=\{s_0,s_1,\dots,s_n\}$, where $s_0:=s_{\theta,1}$ and
$s_i:=s_{\alpha_i}:=s_{\alpha_i,0}$ for $i=1,\dots,n$.
For $s_i,s_j\in S_a$ with $i\neq j$, let $m_{ij}=m_{ji}\in\mathbb{N}_{\ge 2}$ denote the order of $s_is_j$ and set $m_{ii}=1$.

It is also well known that $W_a \simeq \mathbb{Z}\Phi \rtimes W$. Therefore, any element $w \in W_a$ decomposes as $w=\tau_x\overline{w}$ where $x \in \mathbb{Z}\Phi$ and $\overline{w} \in W$. The element $\overline{w}$  is called the \emph{finite part} of $w$. For instance, the affine reflections $s_{\alpha,k}$ decompose as 
\begin{equation}\label{affine reflection decomp}
    s_{\alpha,k} = \tau_{k\alpha}s_{\alpha}.
\end{equation}

Finally, we shall use the following formula,
\begin{equation}\label{conjugacy formula with translation}
    w \tau_{x} w^{-1} = \tau_{w(x)}
\end{equation}
for any $w \in W_a$ and any $x \in \mathbb{Z}\Phi$. It will be convenient to use (\ref{conjugacy formula with translation}) as follows
\begin{equation}\label{conjugacy formula with translation second version}
    w \tau_{x} = \tau_{w(x)}w.
\end{equation}

\medskip

\subsubsection{Alcoves and Shi coefficients}
Let $\alpha \in \Phi$, $k,m \in \mathbb{Z}$. We set the hyperplanes 
\begin{align*}
H_{\alpha,k} & = \{x \in V~|~s_{\alpha,k}(x)=x \} \\
                     &= \{ x \in V~|~ \langle x, \alpha^{\vee} \rangle = k\}
\end{align*}

 and the strips 
 $$
 H_{\alpha,k}^p  = \{x \in V~|~k < \langle x ,\alpha^{\vee} \rangle < k+p \}.
 $$
 An alcove of $V$ is by definition  a connected component of
 
\begin{equation}
 V ~\backslash \bigcup\limits_{\begin{subarray}{c}
 ~ ~\alpha \in \Phi^{+} \\ 
  k \in \mathbb{Z}
\end{subarray}}
H_{\alpha,k}.
\end{equation}

We denote by $A_e$ the alcove defined by $A_e = \bigcap_{\alpha \in \Phi^+} H_{\alpha,0}^1$. The group $W_a$ acts simply transitively on the set of alcoves, hence there is a bijection between $W_a$ and the set of alcoves. This bijection is defined by $w \mapsto A_w$ where $A_w := wA_e$. We call $A_w$ the corresponding alcove associated to $w \in W_a$. Any alcove of $V$ can be written as an intersection of special strips, that is there exists a $\Phi^+$-tuple of integers $(k(w,\alpha))_{\alpha \in \Phi^+}$, that we call the \textit{Shi coefficients} of $w$, such that 
\begin{equation}
A_w = \bigcap\limits_{\alpha \in \Phi^+}H_{\alpha,\, k(w,\alpha)}^1.
\end{equation}

In \cite{JYS1} Jian Yi Shi shows that the $\Phi^+$-tuple of integers $(k(w,\alpha))_{\alpha \in \Phi^+}$ subject to certain conditions characterizes entirely $w$. We call this $\Phi^+$-tuple the \textit{Shi vector} of $w$. Based on this characterization, the author defined in \cite{NC1} an affine variety $\widehat{X}_{W_a}$, called the Shi variety of $W_a$, whose integral points are in bijection with $W_a$. 

We denote by $H^0(\widehat{X}_{W_a})$ the set of irreducible components of $\widehat{X}_{W_a}$. These components are affine subspaces which are parameterized by a collection of integral vectors $\lambda \in \mathbb{N}^m$ that we call admitted vectors. The component associated to an admitted vector $\lambda$ is denoted by $X_{W_a}[\lambda]$ and if $\lambda \neq \gamma$ are both admitted then $X_{W_a}[\lambda] \cap X_{W_a}[\gamma]= \emptyset$ (we recall these notions in Section \ref{background Shi var}). Thus, each element $w \in W_a$ can be seen as an integral point of $\widehat{X}_{W_a}$, which will be denoted by $\iota(w):=(k(w,\alpha))_{\alpha \in \Phi^+}$,  and lies in a specific component (see Theorem \ref{TH central} for more details about the map $\iota$). 

\subsubsection{Fundamental parallelepiped}\label{section parallelepiped}
We define $P_{\Delta} := \bigcap_{\alpha \in \Delta}H_{\alpha,0}^1$. The fundamental weights $\omega_i$ are some of the vertices  of $P_{\Delta}$ and we have 
\begin{equation}
    P_{\Delta} = \left\{\sum\limits_{i=1}^nc_i\omega_i~|~ c_i \in \llbracket 0,1 \rrbracket \right\}.
\end{equation}
 
It is convenient to introduce the following definition
\begin{equation}
    \text{Alc}(P_{\Delta}) := \{ w \in W_a~|~A_w \subset P_{\Delta} \},
\end{equation}
where it is known \cite[\S11-6, Lemma C]{BOURB} that 
\begin{equation}\label{number alcove funda para}
    |\text{Alc}(P_{\Delta})| = \frac{|W|}{f_{\Phi}}.
\end{equation}

\bigskip

\begin{figure}[h!]
\centering
\begin{tabular}{cc} 
Type $A_2$
& 
Type $B_2$
\\
\begin{tikzpicture}[scale=1.08]
\tiny
\clip (-3,-4) rectangle (4,4);
\path[fill=gray!30] (0,0) -- +(0:1) -- ++(60:1);
\path[fill=gray!30] (1,0) -- (0.5,{sqrt(3)/2}) -- (1.5,{sqrt(3)/2}) -- cycle;

\draw[line width = 0.3mm, ->] (0,0) -- + (-30:{sqrt(3)});
\draw[line width = 0.3mm, ->] (0,0) -- + (90:{sqrt(3)});
\draw[line width = 0.3mm, ->] (0,0) -- + (30:{sqrt(3)});
\node[anchor = north west, scale=1.5] at ( 1.5, {-sqrt(3)/2}  ) {$\alpha_1$};
\node[anchor = south, scale=1.5] at ( 0, {sqrt(3)} ) {$\alpha_2$};
\node[anchor = south west, scale=1.5] at ( 1.5, {sqrt(3)/2} ) {$\theta$};
\draw[line width = 0.3mm, ->, color=purple!80] (0,0) -- + (0:1);
\node[anchor = south west, scale=1.5, color=purple!80] at
(1,0) {$\varpi_1$};
\draw[line width = 0.3mm, ->, color=teal] (0,0) -- + (60:1);
\node[anchor = south, scale=1.5, color=teal] at
(0.5,{sqrt(3)/2}) {$\varpi_2$}; 
\foreach \i in {-6,...,6}
{
\draw[color=gray, line width = 0.01mm] (\i,0) -- +(60:6) -- ++(60:-6);
\draw[color=gray, line width = 0.01mm] (-6,{\i*sqrt(3)/2}) -- (6,{\i*sqrt(3)/2});
\draw[color=gray, line width = 0.01mm] ({\i*3/4},{\i*sqrt(3)/4}) -- +(120:6) -- ++(120:-6);
}
\end{tikzpicture}
&  
\begin{tikzpicture}[scale=1.08]
\tiny
\clip (-3,-4) rectangle (4,4);
\path[fill=gray!30] (0,0) -- +(0:1) -- ++(45:{1*sqrt(2)});
\path[fill=gray!30] (1,0) -- (2,0) -- (1,1) -- cycle;
\path[fill=gray!30] (1,1) -- (2,0) -- (2,1) -- cycle;
\path[fill=gray!30] (2,0) -- (2,1) -- (3,1) -- cycle;

\foreach \i in {-6,...,6}
{
\draw[color=gray, line width = 0.01mm] (\i,-\i) -- + (45:10) -- ++ (45:-10);
\draw[color=gray, line width = 0.01mm] (\i,\i) -- + (-45:10) -- ++ (-45:-10);
\draw[color=gray, line width = 0.01mm] (0,\i) -- + (0:10) -- ++ (0:-10);
\draw[color=gray, line width = 0.01mm] (\i,0) -- + (90:10) -- ++ (90:-10);
}

\draw[line width = 0.3mm, ->] (0,0) -- + (-45:{2*sqrt(2)});
\node[anchor = north west, scale=1.5] at
(2,-2) {$\alpha_1$};

\draw[line width = 0.3mm, ->] (0,0) -- + (90:2);

\node[anchor = south, scale=1.5] at
(0,2) {$\alpha_2$};
\draw[line width = 0.3mm, ->] (0,0) -- + (0:2);
\node[anchor = north west, scale=1.5] at
(2,0) {$\theta$};
\draw[line width = 0.3mm, ->, color=purple!80] (0,0) -- + (0:1.97);
\node[anchor = south west, scale=1.5, color=purple!80] at
(2,0) {$\varpi_1$};
\draw[line width = 0.3mm, ->, color=teal] (0,0) -- + (45:{sqrt(2)});
\node[anchor = south, scale=1.5, color=teal] at
(1,1) {$\varpi_2$}; 
\end{tikzpicture} 
\end{tabular}
\caption{The fundamental parallelepipeds in type $A_2$ and $B_2$. We have $\text{Alc}(P_{A_2})=\{e, \, s_0\}$ and $\text{Alc}(P_{B_2})=\{e, \,s_0,\,s_0s_1, \, s_0s_1s_2\}$. The simple roots, the fundamental weights and the highest short root $\theta$ are also drawn.}
\end{figure}

 \medskip

\subsection{Goals of this article} The first goal of this paper is to study the behavior, up to conjugacy, of the sections of the exact sequence~(\ref{eq seq Z}) when the coefficient space $ \mathbb{Z}\Phi $ is replaced by the $A$-module $A\Phi$, where $ A $ is a commutative ring satisfying $ \mathbb{Z} \subset A \subset \mathbb{R} $.

\medskip

The second goal is to find the smallest $A$ (for inclusion) such that $H^1(W, A\Phi)=0$. This is done in Section \ref{coho in A}. We say that there is an $A\Phi$-obstruction between two sections $\mathsf{s}_1$ and $\mathsf{s}_2$ if they do not define the same class in $H^1(W, A\Phi)$. 
Section \ref{coho in A} also expresses the $A\Phi$-conjugacy in terms of the transpose of the Cartan matrix. We explain how this matrix detects the obstructions between sections. In particular we give a general tool (Proposition \ref{bij coh}) that enables us to give in Section \ref{concrete} some modular equations describing the elements of $H^1(W,\mathbb{Z}\Phi)$ in type $A,B,C$ and $D$.

\medskip

The third and main goal of this article, addressed in Section~\ref{Sec orientation alcoves}, is to express the orientation of an alcove in terms of its Shi vector.
Let us first recall what is meant by \emph{orientation}. Two alcoves $A_w$ and $A_{w'}$ in $W_a$ have the same orientation if one is a translation of the other; that is, if there exists $x \in \mathbb{Z}\Phi$ such that $w' = \tau_x w$ (we give more details in Section \ref{Sec orientation alcoves}).

While this operation is elementary at the group level, it becomes far from trivial when viewed in terms of the Shi vectors. Since studying affine Weyl groups via Shi vectors has been particularly successful, leading for example to a complete description of Kazhdan-Lusztig cells in type $A_n$, it is natural to investigate orientation from this perspective. To deepen our understanding of orientation from this perspective, we will characterize it in terms of the Shi variety and the first cohomology group $H^1(W, \mathbb{Z}\Phi)$.

To this end, we associate to each element $w \in W_a$ a section $\mathsf{s}_w$ of a natural short exact sequence, which, combined with the results of Section~\ref{coho in A}, leads to Theorem~\ref{coho eq orientation}.
Using this theorem along with the results of Section~\ref{concrete}, we derive the modular equations that characterize the orientations of the alcoves.
To motivate the reader, we conclude this section with a reformulation of Theorem~\ref{coho eq orientation} and Corollary~\ref{coro orientation} in type $A_n$, as follows:

\begin{theorem}
Let $w, w' \in W_a$.  Then $A_w$ and $A_{w'}$ have the same orientation if and only if the following two points are satisfied:
\begin{itemize}
\item[(i)]  The Shi vectors $\iota(w)$ and $\iota(w')$ are in the same irreducible component of $\widehat{X}_{W_a}$.
\item[(ii)]  The sections $\mathsf{s}_{w}$ and $\mathsf{s}_{w'}$ define the same element in $H^1(W,\mathbb{Z}\Phi)$.
\end{itemize}
\end{theorem}

\begin{corollary}
Assume that $W=W(A_n)$. Let $w,w' \in W_a$. Then $A_w$ and $A_{w'}$ have the same orientation if and only if  one has in $\mathbb{Z}/(n+1)\mathbb{Z}$ the following equality (where $\overline{k}(w,\alpha):= \overline{k(w,\alpha)}$)
\begin{equation}
    \sum\limits_{j=1}^n j\overline{k}(w,\alpha_j) = \sum\limits_{j=1}^n j\overline{k}(w',\alpha_j).
\end{equation}
\end{corollary}

\section{Background about the Shi variety}\label{background Shi var}

 We recall in this section some necessary material. If $\mathcal{I}$ is an ideal of a polynomial ring $k[X_1,\dots, X_r]$ we denote $V(\mathcal{I})=\{x \in k^r~|~P(x)=0~\forall~P \in \mathcal{I}\}$.
 All the following definitions were introduced in \cite[Section 4]{NC1}. We denote $\mathbb{Z}[X_{\Delta}] := \mathbb{Z}[X_{\alpha_1},\dots ,X_{\alpha_n}]$ and $\mathbb{Z}[X_{\Phi^+}] := \mathbb{Z}[X_{\beta_1},\dots ,X_{\beta_m}]$. For $w \in W_a$ and $Q \in \mathbb{Z}[X_{\Delta}]$ we denote 
\begin{equation}
   Q(w):=Q(k(w,\alpha_1),\dots ,k(w,\alpha_n)). 
\end{equation}

For short, when we need to involve the simple Shi coefficients $k(w,\alpha_i)$, we will often write $Q(w) = Q(\{k(w,\alpha_i)\},~\alpha_i \in \Delta) $ or just $Q(\{k(w,\alpha_i)\})$ if there is no confusion. For a concrete example of the notions involved above, see~\cite[Example~4.2]{NC1}.

\medskip

The following theorem is Shi's characterization of the elements $w \in W_a$ by their $\Phi^+$-tuples of integers.
\begin{theorem}[{\cite[Theorem 5.2]{JYS1}}]\label{thJYS1} 
Let $A = \bigcap\limits_{\alpha \in \Phi^+} H^1_{\alpha,k_{\alpha}}$ with $k_{\alpha} \in \mathbb{Z}$. Then $A$ is an alcove, if and only if, for all $\alpha$, $\beta \in \Phi^+$ satisfying  $\alpha + \beta \in \Phi^+$, we have the following inequality
$$
||\alpha||^2k_{\alpha} + ||\beta||^2k_{\beta} +1 \leq ||\alpha + \beta||^{2}(k_{\alpha+\beta} +1) \leq ||\alpha||^2k_{\alpha} + ||\beta||^2k_{\beta} + ||\alpha||^2+ ||\beta||^2 + ||\alpha+\beta||^2 -1.
$$

\end{theorem}

\begin{remark}
In type $A$ Theorem \ref{thJYS1} has an easy reformulation: Let $A = \bigcap_{\alpha \in \Phi^+} H^1_{\alpha,k_{\alpha}}$ with $k_{\alpha} \in \mathbb{Z}$. Then $A$ is an alcove, if and only if, for all $\alpha$, $\beta \in \Phi^+$ satisfying  $\alpha + \beta \in \Phi^+$, we have the following inequality
\begin{equation}\label{Shi ineq A}
k_{\alpha} + k_{\beta}  \leq k_{\alpha+\beta}  \leq k_{\alpha} + k_{\beta} +1.
\end{equation}
\end{remark}

The following theorem decomposes the Shi coefficients as polynomial equations.
\begin{theorem}[{\cite[Th. 4.1, Lem. 4.1]{NC1}}]\label{polynome}
Let $w \in W_a$. Then for all $\beta \in \Phi^+$ there exists a linear polynomial $P_{\beta} \in \mathbb{Z}[X_{\Delta}]$  with positive coefficients and  $\lambda_{\beta}(w) \in \llbracket 0, \mathsf{h}(\beta^{\vee})-1\rrbracket$ such that 
\begin{equation} \label{eq:P}
k(w,\beta) = P_{\beta}(w) + \lambda_{\beta}(w).
\end{equation}
Moreover, the polynomial $P_{\beta}$ satisfies 
$$
\beta^{\vee} = P_{\beta}(\alpha_1^{\vee},\dots ,\alpha_n^{\vee}).
$$ 

\end{theorem}

\begin{definition} Let $\beta \in \Phi^+$. Write $I_{\beta} := \llbracket 0, \mathsf{h}(\beta^{\vee})-1 \rrbracket$. Notice that if $\beta$ is a simple root then $I_{\beta}=\{0\}$. For any root $\beta \in \Delta$ we set $P_{\beta} = X_{\beta}$ and $\lambda_{\beta}=0$. We denote by $P_{\beta}[\lambda_{\beta}]$ the polynomial $P_{\beta}+ \lambda_{\beta}-X_{\beta} \in \mathbb{Z}[X_{\Phi^+}]$. We define the ideal $J_{W_a}$ of $\mathbb{R}[X_{\Phi^+}]$ as ${J_{W_a} := \sum\limits_{\alpha \in \Phi^+} ~\langle \prod\limits_{\lambda_{\alpha} \in I_{\alpha}} P_{\alpha}[\lambda_{\alpha}] \rangle}$. We define the affine variety $X_{W_a}$ to be:
\begin{equation}
      X_{W_a} := V(J_{W_a}).
\end{equation}
\end{definition}

\begin{definition}\label{admissible vector}
 We say that $v=(v_{\alpha})_{\alpha \in \Phi^+} \in \mathbb{N}^m$ is an \emph{admissible vector} (or just \emph{admissible}) if it satisfies the boundary conditions, that is if for all $\alpha \in \Phi^+$ one has $v_{\alpha} \in I_{\alpha}$. For instance, all the $\lambda:=(\lambda_{\alpha})_{\alpha \in \Phi^+}$ coming from the polynomials $P_{\alpha}[\lambda_{\alpha}]$ give rise to admissible vectors. Furthermore, each admissible vector arises this way. We will write $\lambda$ instead of $(\lambda_{\alpha})_{\alpha \in \Phi^+}$.
\end{definition}

 \begin{definition}
   Let $\lambda$ be an admissible vector. We denote 
  $$
  J_{W_a}[\lambda] := \sum\limits_{\alpha \in \Phi^+} \langle P_{\alpha}[\lambda_{\alpha}] \rangle = \langle P_{\alpha}[\lambda_{\alpha}],~\alpha \in \Phi^+\rangle,
 $$ 			
$$
X_{W_a}[\lambda] := V(J_{W_a}[\lambda]).
$$
 \end{definition}

 \begin{definition}\label{admitted vector}
We will denote $S[W_a]$ as the system of all the inequalities coming from Theorem \ref{thJYS1}. Let $\lambda$ be an admissible vector. We say that $\lambda$ is \emph{admitted} if it satisfies the system $S[W_a]$. Let $w \in W_a$. For an admitted vector $\lambda$, we write $\lambda(w) = \lambda$ if and only if $\iota(w) \in \widehat{X}_{W_a}[\lambda]$. See also \cite[Examples 4.2, 4.3, 4.4]{NC1} for good examples.
\end{definition}

If $Y \subset \mathbb{R}^m$ we denote by $Y(\mathbb{Z})$ the set of integral points of $Y$. Finally, we have the following result which gives the parameterization of the elements of $H^0(\widehat{X}_{W_a})$ in terms of the admitted vectors.

\begin{theorem}[{\cite[Theorem 4.3]{NC1}}]\label{TH central}
The map $\iota : {W_a \longrightarrow X_{W_a}(\mathbb{Z})}$ defined by $w \longmapsto (k(w,\alpha))_{\alpha \in \Phi^+}$ induces by corestriction a bijective map from $W_a$ to the integral points of a subvariety of $X_{W_a}$, denoted $\widehat{X}_{W_a}$, which we call the Shi variety of $W_a$. This subvariety decomposes as follows $$\widehat{X}_{W_a} = \bigsqcup\limits_{\lambda~\text{admitted}} X_{W_a}[\lambda].$$
\end{theorem}

\section{Cohomology of Weyl groups}\label{section cohomo}
In this section we show how the irreducible components of the Shi variety together with the first  cohomology group of $W$ can be used to express the orientation of the alcoves.

Let  $S=\{s_1,\dots ,s_n\}$ be the set of Coxeter generators of $W$ associated to $\Delta$ and $m_{ij}:=\text{ord}(s_is_j)$. 
 The Cartan matrix $C_{\Phi}=(h_{ij})_{i,j}$ associated to $\Phi$ is the matrix of $\mathrm{GL}_n(\mathbb{R})$ defined by $h_{ij} = \langle \alpha_i, \alpha_j^{\vee} \rangle$. It is well known that $C_{\Phi}$ is invertible since it is a change of basis matrix.

\medskip

The decomposition $W_a \simeq \mathbb{Z}\Phi \rtimes W$ induces the following short exact sequence 
\begin{equation}\label{eq seq Z}
\xymatrix{
1  \ar[r]  & \mathbb{Z}\Phi \ar[r]^{i~\text{~~}~\text{~~}} & \mathbb{Z}\Phi \rtimes W \ar[r]^{~\text{~~}\text{~~}~~\pi} & W \ar[r] & 1,
}
\end{equation}

 where $i$ and $\pi$ are the morphisms defined by $i(x) =\tau_x$ and $\pi(w) = \overline{w}$.
 
\subsection{Background about group cohomology}\label{back coho}

Let $A$ be a $G$-module.
We define $A^G$ to be the submodule of $G$-invariants, that is the subgroup of $A$ defined by $A^G := \{ a \in A~|~ga=a ~\text{for all} ~g \in G \}$. The functor $F : A \mapsto A^G$ from the category of $G$-modules to the category of abelian groups is covariant and left-exact. Therefore, we can define its right derived functors $R^iF$ and by definition the $i$-th cohomology group of $G$ with coefficients in $A$ is $H^n(G, A) := R^nF(A)$.

Another description of group cohomology in degree 1 is given by 1-cocycles and 1-coboundaries. A $n$-cochain is  a map $f : G^n \rightarrow A $ and the set of $n$-cochains is denoted by $C^n(G,A)$. For $n=0$ we have $C^0(G,A) = A$ and we denote by $f_a$ ($a \in A$)  the elements of $C^0(G,A)$. The set $C^n(G,A)$, endowed with the addition induced from $A$, is a group.

 The boundary map  $\partial_0 : C^0(G,A) \rightarrow C^{1}(G,A)$ is defined by $\partial_0f_a(x) = xa-a$ and  the boundary map $\partial_1 : C^1(G,A) \rightarrow C^{2}(G,A)$ is defined by $\partial_1 f(x,y) = xf(y)-f(xy)+f(x)$. The 1-cocycles are the elements $f \in C^1(G,A)$ satisfying $\partial_1 f = 0$, and the 1-coboundaries are the elements $f \in C^1(G,A)$ such that $f=\partial_0 f_a$ for some $a \in A$. The set of 1-cocycles is denoted by $Z^1(G,A)$ and the set of 1-coboundaries by $B^1(G,A)$. These two sets are subgroups of $C^1(G,A)$ and $B^1(G,A) \subset Z^1(G,A)$. 
 
 Therefore, we can define $H^1(G,A)$ by $$H^1(G,A) = Z^1(G,A) / B^1(G,A).$$

\medskip

Let $B$ be a $G$-module, i.e. an abelian group endowed with an action of $G$ by automorphisms; equivalently, a morphism $\varphi:G\rightarrow \text{Aut}(B)$, where the action of $G$ is given by $g\cdot b:=\varphi(g)(b)$. Let us consider the following exact sequence:
$$
\xymatrix{
1  \ar[r]  & B \ar[r]^{i~\text{~~}~\text{~~}} & B \rtimes_{\varphi} G \ar[r]^{~~\text{~}~\text{~}~\pi} & G \ar[r] & 1.
}
$$

A section of $\pi$ is a morphism $\mathsf{s} : G \rightarrow B 
    \rtimes_{\varphi} G$ such that $\pi \circ \mathsf{s} = id_G$. We denote by $Sec_B(\pi)$ the set of sections of $\pi$. If $B=\mathbb{Z}\Phi$ we will just write $Sec(\pi)$. 

Two sections $\mathsf{s}$ and $\mathsf{s}'$ are called $B$-conjugate if there exists $b \in B$ such that for all $g \in G$ we have $\mathsf{s}'(g) = i(b)\mathsf{s}(g)i(b)^{-1}$. We denote this by $\mathsf{s} \thicksim_B\mathsf{s}'$, or just $\mathsf{s} \thicksim  \mathsf{s}'$ if there is no confusion.

The following result is also well known and will be of interest for us (see \cite[Theorem 4.4]{JPS},  for a good reference).

\begin{theorem}\label{bij H1}
The set of sections up to $B$-conjugacy of the exact sequence 
$$
\xymatrix{
1  \ar[r]  & B \ar[r]^{i~\text{~~}~\text{~~}} & B \rtimes_{\varphi} G \ar[r]^{~~\text{}~\text{~}~\pi} & G \ar[r] & 1
}
$$
is in bijection with $H^1(G,B)$.
\end{theorem}

\subsection{\texorpdfstring{Cohomology in degree 1 with coefficients in $A\Phi$}{Cohomology in degree 1 with coefficients in AΦ}}\label{coho in A}

This section is dedicated to the study of the cohomology in degree 1 of a Weyl group $W$ with coefficients in $A\Phi$, where $A\Phi$ is the $A$-module generated by $\Phi$ in $V$. Since $\Phi$ is crystallographic, we have $A\Phi = A\Delta$.

Let us set once for all $A$ to be a commutative unitary ring such that $\mathbb{Z} \subset A \subset \mathbb{R}$. Throughout this section we will use the following short exact sequence
\begin{equation}\label{ex seq}
\xymatrix{
1  \ar[r]  & A\Phi \ar[r]^{i~\text{~~}~\text{~~}~~} & A\Phi \rtimes W \ar[r]^{\text{~~}~~\pi} & W \ar[r] & 1.
}
\end{equation}

We state now the main results of this section:

\begin{proposition}\label{bij coh}
Let $\mathsf{s}_1$ and $\mathsf{s}_2$ be two different sections of  (\ref{ex seq}) defined by $\mathsf{s}_1(s_i) = \tau_{x_i}s_i$ and $\mathsf{s}_2(s_i) = \tau_{y_i}s_i$ with $x_i, y_i \in A\Phi$ for $i=1,\dots ,n$. Then the following are equivalent:
\begin{itemize}
\item[(i)] $\mathsf{s}_1 \thicksim \mathsf{s}_2$.
\item[(ii)] There exists $z \in A\Phi$ such that $x_i-y_i=\langle z, \alpha_i^{\vee}\rangle \alpha_i$ for any $i= 1,\dots, n$.
\end{itemize}
\end{proposition}

\begin{theorem}\label{Cartan A}
The following points are equivalent:
\begin{itemize}
\item[(i)] $H^1(W,A\Phi) = 0$.
\item[(ii)] $C_{\Phi} \in \mathrm{GL}_n(A)$.
\end{itemize}
\end{theorem}

\begin{theorem}\label{A minimal}
 Let $f_{\Phi}$ be the index of connection of $\Phi$ and $\mathbb{Z}_{f_{\Phi}} := \mathbb{Z}[\frac{1}{f_{\Phi}}]$. Then the following points are equivalent 
 \begin{itemize}
 \item[(i)] $H^1(W, A\Phi) = 0$ with $A$ minimal.
 \item[(ii)]$A=\mathbb{Z}_{f_{\Phi}}$.
 \end{itemize}
\end{theorem}

\begin{example}\label{big example}
Let us first illustrate this with an example. Let us take $W=W(B_2)$. The Cartan matrix of $B_2$ and its inverse are
$$
C_{B_2}=
\begin{pmatrix} 
    2 & -2    \\
    -1  &  2   \\
\end{pmatrix}
~~\text{and}~~
{}^{t}C_{B_2}^{-1}=\frac{1}{2}
\begin{pmatrix} 
    2 & 1  \\
    2  &  2   \\
\end{pmatrix}.
$$
Let $\mathsf{s}_1$ and $\mathsf{s}_2$ be two sections of $W(B_2)$. We know by Lemma \ref{braid lemma} that we can identify $\mathsf{s}_1$ with a point $(x_1, x_2) \in \mathbb{Z}\alpha_1 \times \mathbb{Z}\alpha_2$ and $\mathsf{s}_2$ with a point $(y_1, y_2) \in \mathbb{Z}\alpha_1 \times \mathbb{Z}\alpha_2$. Denote ${x_i = a_i\alpha_i}$ and $y_i = b_i\alpha_i$. Moreover we know that 
\begin{align*}
    \mathsf{s}_1 \thicksim \mathsf{s}_2 & \Longleftrightarrow \exists z \in \mathbb{Z}\alpha_1 \times \mathbb{Z}\alpha_2 ~\text{such that}~ \left\{
\begin{array}{ll}
x_1-y_1= \langle z,\alpha_1^{\vee}\rangle\alpha_1\\
x_2-y_2= \langle z,\alpha_2^{\vee}\rangle\alpha_2
\end{array}
\right. \\
& \Longleftrightarrow \exists z \in \mathbb{Z}\alpha_1 \times \mathbb{Z}\alpha_2 ~\text{such that}~ \left\{
\begin{array}{ll}
a_1\alpha_1-b_1\alpha_1= \langle z,\alpha_1^{\vee}\rangle\alpha_1\\
a_2\alpha_2-b_2\alpha_2= \langle z,\alpha_2^{\vee}\rangle\alpha_2
\end{array}
\right. \\
& \Longleftrightarrow \exists z \in \mathbb{Z}\alpha_1 \times \mathbb{Z}\alpha_2 ~\text{such that}~ \left\{
\begin{array}{ll}
a_1-b_1= \langle z,\alpha_1^{\vee}\rangle\\
a_2-b_2= \langle z,\alpha_2^{\vee}\rangle.
\end{array}
\right. 
\end{align*}

It follows then that 
\begin{align*}
\mathsf{s}_1 \thicksim \mathsf{s}_2 & \Longleftrightarrow \exists z \in \mathbb{Z}\alpha_1 \times \mathbb{Z}\alpha_2 ~\text{such that}~ 
{}^{t}C_{B_2}
\begin{pmatrix} 
      z_1  \\
    z_2     \\
\end{pmatrix}
=\begin{pmatrix} 
 a_1-b_1 \\
  a_2-b_2\\
\end{pmatrix} \\
&  \Longleftrightarrow \exists z \in \mathbb{Z}\alpha_1 \times \mathbb{Z}\alpha_2 ~\text{such that}~ 
\begin{pmatrix} 
      z_1  \\
    z_2     \\
\end{pmatrix}
=\frac{1}{2}
\begin{pmatrix} 
    2 & 1  \\
    2  &  2   \\
\end{pmatrix}\begin{pmatrix} 
 a_1-b_1 \\
  a_2-b_2\\
\end{pmatrix}.
\end{align*}

Thus, we must have $z_1$ and $z_2$ both integral, and such that  
$$
\left\{
 \begin{array}{ll}
2z_1 =  2(a_1-b_1)+2(a_2-b_2)  \\
2z_2 =   2(a_1 - b_1) +a_2-b_2.
\end{array}
\right. 
$$

These arithmetic conditions show how $H^1(W(B_2), \mathbb{Z}\Phi)$ behaves, and we see in this particular situation that we do not have $H^1(W(B_2), \mathbb{Z}\Phi) = 0$. Indeed, by taking for example $a_1=b_1=1$, $a_2=3$ and $b_2=2$ we have $2z_2=1$, which is impossible since $z_2$ is integral. We thus have an obstruction between $\mathsf{s}_1$ and $\mathsf{s}_2$, which implies that $H^1(W(B_2),\mathbb{Z}\Phi)\neq 0$.

However, if we consider the equations in the ring $\mathbb{Z}_2=\mathbb{Z}[\frac{1}{2}]$ instead of $\mathbb{Z}$, the previous obstruction no longer exists and it follows that $H^1(W(B_2), \mathbb{Z}_2\Phi) = 0$. Furthermore, the ring $\mathbb{Z}_2$ is the smallest one that satisfies the last equality.

\end{example}

\begin{lemma}\label{braid lemma}
Let $\mathsf{s}$ be a section of the exact sequence (\ref{ex seq}) defined by $\mathsf{s}(s_i)=\tau_{x_i}s_i$ for $s_i \in S$ and some $x_i \in A\Phi$. Then the  map $\xi $ from  $Sec_{A\Phi}(\pi)$ to $A\alpha_1 \times \cdots \times A\alpha_n$ defined by ${\xi(\mathsf{s}) = (x_1,\dots ,x_n)}$ is a bijection.
\end{lemma}

\begin{proof}

The group $W$ is a Coxeter group given by generators and relations, so $W = \langle S~|~R \rangle$ where $S = \{s_1,\dots, s_n\}$ and $R = \{(s_is_j)^{m_{ij}}~|~i,j = 1\dots,n\}$. For example, in type $A_n$, the relations are given by $R = \{s_i^2~|~i = 1,\dots,n\} \sqcup \{(s_is_j)^2~|~|i-j| > 1\} \sqcup \{(s_is_{i+1})^3~|~1 \leq i \leq n-1\}$. 
Therefore, $W = F(S)/\langle \langle R \rangle \rangle$ where $F(S)$ is the free group over $S$ and  $\langle \langle R \rangle \rangle$ is the normal closure of $R$. 

\bigskip

$\bullet$ Let us first show that $\xi$ is well defined, that is $x_i \in A\alpha_i$ for all $i$. Since $\mathsf{s}$ is a morphism, it must preserve all the relations of $W$. Therefore, the relation $s_i^2=e$ implies that $\tau_{x_i}s_i\tau_{x_i}s_i = e$, and by (\ref{conjugacy formula with translation second version}) we have  $\tau_{x_i}\tau_{s_i(x_i)}s_is_i = e$, that is
$\tau_{x_i + s_i(x_i)}s_i^2 = e$. Hence $x_i + s_i(x_i) = 0$, and then we must have $s_i(x_i) = -x_i$, which means that ${x_i \in H_{\alpha_i}^{\perp} = \mathbb{R}\alpha_i}$. Thus, for all $i=1,\dots ,n$ there exists $k_i \in \mathbb{R}$ such that $x_i=k_i\alpha_i$. However, because $x_i \in A\Phi$ and $A\Phi = A\Delta$, we in fact have $k_i \in A$. 

\bigskip

$\bullet$  Conversely, we claim that any choice of $\mathsf{k} =(k_1,\dots,k_n)\in A^n$ gives rise to a section $\mathsf{s}_{\mathsf{k}}$. From this claim it is clear that the map $\xi$ is a bijection. 

 \noindent\textit{Proof of the claim}. Let us define the map $f_{\mathsf{k}} : S \rightarrow A\Phi \rtimes W$  by $f_{\mathsf{k}}(s_i) = \tau_{k_i\alpha_i}s_i$.
  From the universal property of free groups, the map $f_{\mathsf{k}}$ extends uniquely into a morphism   $\widetilde{f_{\mathsf{k}}} : F(S) \rightarrow A\Phi \rtimes W$ such that the following diagram commutes
\[
\xymatrix@R=0.8pc@C=2.2pc{
  & F(S) \ar[dd]^{\widetilde{f}_{\mathsf{k}}} \\
  S ~ \ar@{^{(}->}[ru] \ar[rd]_{\,\hspace{-2em}f_{\mathsf{k}}} & \\
  & A\Phi \rtimes W.
}
\]
Let us show now that $\langle \langle R \rangle \rangle \subset \text{ker}(\widetilde{f_{\mathsf{k}}})$. It suffices to show that $R \subset \text{ker}(\widetilde{f_{\mathsf{k}}})$. Recall the notation $s_i = s_{\alpha_i}$ and also that for $s_i, s_j \in S$, the integer $m_{ij}$ is the order of $s_is_j$ in $W$. In order to show the inclusion $R \subset \text{ker}(\widetilde{f_{\mathsf{k}}})$ we need to show that $\widetilde{f_{\mathsf{k}}}\big((s_is_j)^{m_{ij}}\big) = e$, where $(s_is_j)^{m_{ij}}$ is seen here as element of $F(S)$, meaning in particular that it is a priori not equal to the identity element. For any $\alpha_i, \alpha_j \in \Delta$, we have $\text{ord}(s_{\alpha_i,k_i}s_{\alpha_j,k_j}) = \text{ord}(s_{\alpha_i}s_{\alpha_j}) = m_{ij}$ \cite[Proposition 1.3.4]{NathanChapelierPhdThesis}. However, by (\ref{affine reflection decomp}) we have $s_{\alpha_i,k_i}=\tau_{k_i\alpha_i}s_{\alpha_i}=\widetilde{f_{\mathsf{k}}}(s_i)$. Therefore, 
$$
\widetilde{f_{\mathsf{k}}}\big((s_is_j)^{m_{ij}}\big) = \Big( \widetilde{f_{\mathsf{k}}}(s_is_j)\Big)^{m_{ij}} =\Big( \widetilde{f_{\mathsf{k}}}(s_i)\widetilde{f_{\mathsf{k}}}(s_j)\Big)^{m_{ij}}  = \Big( s_{\alpha_i,k_i}s_{\alpha_j,k_j}\Big)^{m_{ij}}=e.
$$
Thus, the map $\widetilde{f_{\mathsf{k}}}$ factors through the quotient $W = F(S)/\langle \langle R \rangle \rangle$ into a morphism denoted $\mathsf{s}_{\mathsf{k}}$, and we can complete the above diagram as the following commutative diagram 

\[
\xymatrix@R=0.8pc@C=2.2pc{
  & F(S) \ar[dd]^{\widetilde{f}_{\mathsf{k}}} \ar@{->>}[rd]\\
  S ~ \ar@{^{(}->}[ru] \ar[rd]_{\,\hspace{-2em}f_{\mathsf{k}}} & & W. \ar[ld]^{\,\hspace{1em}\mathsf{s}_{\mathsf{k}}}\\
  & A\Phi \rtimes W
}
\]

 Finally, it is obvious that $\pi \circ \mathsf{s}_{\mathsf{k}} = id_{W}$.

\end{proof}

\newpage

\begin{proof}[Proof of Proposition \ref{bij coh}]
 Since $\mathsf{s}_1$ and $\mathsf{s}_2$ are different there exists at least one index $j$ such that ${x_j \neq y_j}$. By definition, $\mathsf{s}_1$ is $A\Phi$-conjugate to $\mathsf{s}_2$ if and only if there exists $z \in A\Phi$ such that for all $w \in W$ one has ${\mathsf{s}_2(w)=\tau_z\mathsf{s}_1(w)\tau_{-z}}$. It turns out that ${\tau_z\mathsf{s}_1(w)\tau_{-z} = \tau_{(id-w)(z)}\mathsf{s}_1(w)}$, indeed since $\mathsf{s}_1(w) \in A\Phi \rtimes W$ there exists $u \in A\Phi$ such that $\mathsf{s}_1(w)=\tau_uw$. Therefore we have  
 \begin{align*}
  \tau_z\mathsf{s}_1(w)\tau_{-z} = \tau_z\tau_uw\tau_{-z} &=\tau_z\tau_u\tau_{-w(z)}w = \tau_z\tau_{-w(z)}\tau_uw 
 = \tau_{(id-w)(z)}\mathsf{s}_1(w).
 \end{align*}

 Hence, it follows
\begin{align*}
\mathsf{s}_1 \thicksim \mathsf{s}_2 ~~~ \Longleftrightarrow & ~~~\exists z \in A\Phi~\text{such that}~\forall~ w \in W, ~\mathsf{s}_2(w)= \tau_{(id-w)(z)}\mathsf{s}_1(w) \\
														~~~ \Longrightarrow & ~~~\exists z \in A\Phi~\text{such that}~\forall~ s_i \in S, ~\mathsf{s}_2(s_i)= \tau_{(id-s_i)(z)}\mathsf{s}_1(s_i) \\
													    ~~~  \Longleftrightarrow & ~~~\exists z \in A\Phi~\text{such that}~\forall~s_i \in S, ~y_i-x_i=(id-s_i)(z). \\ 
\end{align*}

We also have the opposite direction. Indeed, let us write $w=t_1t_2\dots t_p$ with $t_i \in S$ such that it is a reduced expression. Via the following equalities we have
\begin{align*}
\mathsf{s}_1(w)  &= \mathsf{s}_1(t_1)\mathsf{s}_1(t_2)\dots \mathsf{s}_1(t_p) \\
&= \tau_{x_1}t_1\tau_{x_2}t_2\dots \tau_{x_p}t_p \\
&=\tau_{x_1+t_1(x_2)+\cdots+t_1t_2\dots t_{p-1}(x_p)}t_1t_2\dots t_p,
\end{align*}
and
\begin{align*}
\mathsf{s}_2(w)  &= \mathsf{s}_2(t_1)\mathsf{s}_2(t_2)\dots \mathsf{s}_2(t_p) \\
&= \tau_{y_1}t_1\tau_{y_2}t_2\dots \tau_{y_p}t_p \\
&=\tau_{y_1+t_1(y_2)+\cdots+t_1t_2\dots t_{p-1}(y_p)}t_1t_2\dots t_p.
\end{align*}

It follows that
\begin{align*}
																															 &  \mathsf{s}_2(w) =  \tau_{(id-w)(z)}\mathsf{s}_1(w) \\
																					  \Longleftrightarrow   ~~   & \tau_{y_1+t_1(y_2)+\cdots+t_1t_2\dots t_{p-1}(y_p)}t_1t_2\dots t_p =  \tau_{(id-w)(z)}\tau_{x_1+t_1(x_2)+\cdots+t_1t_2\dots t_{p-1}(x_p)}t_1t_2\dots t_p  \\
 																						 \Longleftrightarrow  ~~ & y_1+t_1(y_2)+\cdots+t_1t_2\dots t_{p-1}(y_p) = z-w(z) +  x_1+t_1(x_2)+\cdots+t_1t_2\dots t_{p-1}(x_p) \\
 																					 \Longleftrightarrow    ~~  & z-w(z) = (y_1 - x_1) + t_1(y_2-x_2) + \cdots + t_1t_2\dots t_{p-1}(y_p - x_p).
\end{align*}

However, with the assumption $y_i-x_i = (id-t_i)(z)$ for all $i=1,\dots ,p$,   it follows that
\begin{align*}
z-w(z)  &= z-t_1(z)+t_1(z)-t_1t_2(z)
 +t_1t_2(z)-\cdots -t_1t_2\dots t_{p-1}(z) + t_1t_2\dots t_{p-1}(z)-w(z) \\
                           & = z-t_1(z) + t_1(z-t_2(z)) +\cdots+ t_1t_2\dots t_{p-1}(z-t_p(z)) \\
                           & = y_1-x_1 + t_1(y_2-x_2) + \cdots + t_1t_2\dots t_{p-1}(y_p-x_p).
\end{align*}
Therefore we have the equivalence  
$$
\mathsf{s}_1 \thicksim \mathsf{s}_2 ~ \Longleftrightarrow  ~\exists z \in A\Phi~\text{such that}~\forall~ s_i \in S, ~y_i-x_i=(id-s_i)(z).
$$
We conclude the proof with the following computation
\begin{align*}
    (id-s_i)(z)& = z-s_i(z) \\
               & = z-(z-\langle z, \alpha_i^{\vee}\rangle\alpha_i) \\
               & =\langle z, \alpha_i^{\vee}\rangle \alpha_i.
\end{align*}

We conclude by applying the change of variables $z \mapsto -z$.
\end{proof}

\begin{proof}[Proof of Theorem \ref{Cartan A}]
We first show that (ii) implies (i). By Proposition \ref{bij coh} we know that two sections $\mathsf{s}_1$ and $\mathsf{s}_2$ of (\ref{ex seq}) are $A\Phi$-conjugate if and only if there exists $z \in A\Phi$ such that $\forall$ $i= 1,\dots, n$, one has $x_i-y_i=\langle z, \alpha_i^{\vee}\rangle \alpha_i$. Moreover, by Lemma \ref{braid lemma} we know that $x_i, y_i \in A\alpha_i$, and then $x_i = a_i\alpha_i$ and $y_i = b_i\alpha_i$ for some $a_i, b_i \in A$. This is where the Cartan matrix enters; we omit the details in the below equivalences, which are worked out explicitly in Example~\ref{big example}.
{\small
\begin{align*}
\mathsf{s}_1 \thicksim \mathsf{s}_2 &\Longleftrightarrow \exists z \in A\Phi ~\text{such that}~ \left\{
\begin{array}{ll}
x_1-y_1= \langle z,\alpha_1^{\vee}\rangle\alpha_1\\
 ~~~~~~~~~~~\vdots~~~~~~~~~\\
x_n-y_n= \langle z,\alpha_n^{\vee}\rangle\alpha_n
\end{array}
\right. \\
 &\Longleftrightarrow \exists z \in A\Phi ~\text{such that}~ \left\{
\begin{array}{ll}
a_1-b_1= \langle z,\alpha_1^{\vee}\rangle\\
~~~~~~~~~~~\vdots~~~~~~~~~~\\
a_n-b_n= \langle z,\alpha_n^{\vee}\rangle.
\end{array}
\right. 
\end{align*}
}

Writing $z=z_1\alpha_1 +\cdots+ z_n\alpha_n$ it follows that
{\small
\begin{equation}\label{cartan equation}
\mathsf{s}_1 \thicksim \mathsf{s}_2 \Longleftrightarrow \exists z \in A\Phi ~\text{such that}~~ {}^{t}C_{\Phi} \begin{pmatrix}
z_1\\
\vdots \\
z_n
\end{pmatrix}
= \begin{pmatrix}
a_1-b_1\\
\vdots \\
a_n-b_n
\end{pmatrix}.
\end{equation} 
}

With the assumption: $C_{\Phi} \in \mathrm{GL}_n(A)$, we also have ${}^{t}C_{\Phi}^{-1}\in \mathrm{GL}_n(A)$ and then there is no constraint for the choices of $z_1,\dots, z_n$ in $A$. Therefore, as soon as $C_{\Phi} \in \mathrm{GL}_n(A)$, (\ref{cartan equation}) is always satisfied (in other words there is no obstruction between $\mathsf{s}_1$ and $\mathsf{s}_2$). This implies that for any pair of sections $\mathsf{s}_1, \mathsf{s}_2$, one has $\overline{\mathsf{s}}_1 = \overline{\mathsf{s}}_2$ in $H^1(W, A\Phi)$. Hence, there is only one element in $H^1(W, A\Phi)$, which implies the first direction. 

Let us show now the other direction. Assume that $C_{\Phi} \notin \mathrm{GL}_n(A)$.  Since $C_{\Phi} \in \mathrm{GL}_n(\mathbb{R})$ with its coefficients in $\mathbb{Z}$, the previous assumption implies in particular that ${}^{t}C_{\Phi}^{-1} \notin \mathrm{GL}_n(A)$, and then there exists a coefficient $q_{ij}$ of ${}^{t}C_{\Phi}^{-1}$ such that $q_{ij} \notin A$. Let $\mathsf{s}_1$ be the section with corresponding point $x:=(0,\dots ,1,\dots ,0) \in A\alpha_1\times\cdots \times A\alpha_n$ where 1 is in position $i$, and $\mathsf{s}_2$ be the trivial section. We have that $\mathsf{s}_1 \thicksim \mathsf{s}_2$, if and only if, there exists $z = z_1\alpha_1 +\cdots+ z_n\alpha_n \in A\Phi$ such that ${}^{t}C_{\Phi}z = x-0=x$. Therefore, it follows that 

{\small
$$
\begin{pmatrix}
z_1\\
\vdots\\
z_n
\end{pmatrix}
={}^{t}C_{\Phi}^{-1}
\begin{pmatrix}
0\\
\vdots \\
1 \\
\vdots \\
0
\end{pmatrix}
=
\begin{pmatrix}
q_{i1}\\
\vdots \\
q_{ij} \\
\vdots \\
q_{in}
\end{pmatrix}.
$$
}

Since $z_i \in A$ and $q_{ij} \notin A$, the sections $\mathsf{s}_1$ and $\mathsf{s}_2$ define two different elements in $H^1(W, A\Phi)$. This is impossible since $H^1(W, A\Phi)=0$. Therefore, we must have $C_{\Phi} \in \mathrm{GL}_n(A)$.
\end{proof}

\begin{proof}[Proof of Theorem \ref{A minimal}]
It is well known that the determinant of $C_{\Phi}$ is $f_{\Phi}$ (see \cite{BOURB} Ch. VI, $\S$ 1, exercice 7). Therefore $C_{\Phi}^{-1} = \frac{1}{det(C_{\Phi})}D=\frac{1}{f_{\Phi}}D$ where $D \in \mathrm{M}_n(\mathbb{Z})$. In particular we have $C_{\Phi} \in \mathrm{GL}_n(\mathbb{Z}_{f_{\Phi}})$. 

\medskip

\noindent $\bullet$ The direction $(ii)$ implies $(i)$ is a direct consequence of Theorem \ref{Cartan A}.

\medskip

\noindent $\bullet$ Let us show now the direction $(i)$ implies $(ii)$. Since $H^1(W, A\Phi) = 0$, Theorem \ref{Cartan A} tells us that $C_{\Phi} \in \mathrm{GL}_n(A)$, and then its inverse as well. Therefore, all the coefficients of $C_{\Phi}^{-1}$ are in $A$. However, all the coefficients of $C_{\Phi}^{-1}$ are of the form $\frac{d}{f_{\Phi}}$ for some $d \in \mathbb{Z}$. We claim then that the coefficient $\frac{1}{f_{\Phi}} \in A$. We have two cases to consider, either we are in type $D_n$ or not.

\begin{itemize}
    \item [(a)] In types $A, B, C, E_6, E_7, E_8, F_4$ and $G_2$, a quick inspection (see \cite{YY} together with its appendix for the exceptional cases) of the inverses of the Cartan matrices shows that there is always a coefficient of the form $\frac{d}{f_{\Phi}}$ and another one of the form $\frac{d+1}{f_{\Phi}}$ in the same Cartan matrix. Hence $\frac{d+1}{f_{\Phi}} - \frac{d}{f_{\Phi}} = \frac{1}{f_{\Phi}} \in A$.
    \item[(b)] In type $D_n$, a quick inspection of the inverse of the Cartan matrix shows that it contains both coefficients
$\frac{n-2}{4}$ and $\frac{n}{4}$ (see Section \ref{subD}). This implies in particular that $\frac{n}{4}-\frac{n-2}{4} = \frac{1}{2} \in A$, and since $A$ is a ring $\left(\frac{1}{2} \right)^2  = \frac{1}{4}$ also belongs to $A$. Since $f_{\Phi} = 4$ in type $D_n$, we do have $\frac{1}{f_{\Phi}} \in A$. 
\end{itemize}
Therefore, in each situation we have $1/f_{\Phi} \in A$, which implies that $\mathbb{Z}_{f_{\Phi}} \subset A$. Moreover, since $C_{\Phi} \in \mathrm{GL}_n(\mathbb{Z}_{f_{\Phi}})$, Theorem \ref{Cartan A} tells us that $H^1(W,\mathbb{Z}_{f_{\Phi}}\Phi) = 0$. Since $A$ is supposed to be minimal with respect to the property $H^1(W,A\Phi) = 0$, we must have $A=\mathbb{Z}_{f_{\Phi}}$.
\end{proof}



\subsection{\texorpdfstring{Concrete realization of $H^1(W, \mathbb{Z}\Phi)$}{Concrete realization of H1(W, ZΦ)}}\label{concrete}

Thanks to Lemma \ref{braid lemma}, setting $A=\mathbb{Z}$, we know that the sections of $\pi$ are in bijective correspondence with the elements of $\mathbb{Z}\alpha_1 \times \dots \times \mathbb{Z}\alpha_n$. We also know that the elements of $H^1(W, \mathbb{Z}\Phi)$ are in bijection with the sections of $\pi$ up to $\mathbb{Z}\Phi$-conjugacy. 
Furthermore, the condition of being $\mathbb{Z}\Phi$-conjugate, when considered as a condition on pairs of elements in $\mathbb{Z}\alpha_1 \times \dots \times \mathbb{Z}\alpha_n$, corresponds to the solvability of a system of linear equations defined by the transpose of the Cartan matrix (this is embodied through Proposition \ref{bij coh}).

 In this section we investigate in types $A, B, C$, and $D$ the $\mathbb{Z}\Phi$-conjugacy. 

\medskip
 
  First of all, notice that when the index of connection $f_{\Phi}$ is 1, the Cartan matrix is invertible in $\mathbb{Z},$ which implies via Theorem \ref{A minimal} that $H^1(W, \mathbb{Z}\Phi) = 0$. Hence, if $\Phi$ is of type $G_2$, $F_4$ or $E_8$, the cohomology in degree 1 is trivial.

Let $\mathsf{s}_1$, $\mathsf{s}_2 \in Sec(\pi)$. Let $(x_1,\dots ,x_n) \in  \mathbb{Z}\alpha_1 \times \cdots \times \mathbb{Z}\alpha_n$ be the corresponding point of $\mathsf{s}_1$ and $(y_1,\dots ,y_n)$ be the corresponding point of $\mathsf{s}_2$ (both via Lemma \ref{braid lemma}). Denote $x_i=a_i\alpha_i$ and $y_i=b_i\alpha_i$. We denote by $d_{i}(\mathsf{s}_1, \mathsf{s}_2)$, or just $d_{i}$ if there is no confusion, the number 
\begin{equation}\label{number d_i}
    d_{i} := a_i-b_i.
\end{equation}
 We also define
\begin{equation}
    d:=d_1\alpha_1 +\cdots+ d_n\alpha_n.
\end{equation}

\medskip

\subsubsection{\texorpdfstring{Type $A_n$}{Type An}}\label{subA}
In this section $W=W(A_n)$ and $\Phi = A_n$. We denote by $C:=C_{A_n}$ the Cartan matrix of $A_n$. The coefficients of $C^{-1}$ are known (see \cite{YY}) and are given by 
\begin{equation}
    (C^{-1})_{ij} = \frac{(n+1)\text{min}(i,j) -ij}{n+1},\quad 1 \leq i,j \leq n.
\end{equation}

In particular we see that the inverse of the Cartan matrix in type $A_n$ is symmetric, hence  ${}^{t}C^{-1}= C^{-1}$.

\begin{proposition}\label{H1 dans A}
The sections $\mathsf{s}_1$ and $\mathsf{s}_2$ define two different elements in $H^1(W, \mathbb{Z}\Phi)$ if and only if 
\begin{equation*}
    \overline{d_1} + 2\overline{d_2} +\cdots+ n\overline{d_n} \neq 0 \quad \text{in} \quad \mathbb{Z}/(n+1)\mathbb{Z}.
\end{equation*}

\end{proposition}

\begin{proof}
By (\ref{cartan equation}) we know that $\mathsf{s}_1 \thicksim \mathsf{s}_2$ if and only if there exists $z \in \mathbb{Z}\alpha_1 \times \cdots \times \mathbb{Z}\alpha_n$ such that $z={}^{t}C^{-1}d$, that is if and only if $z=C^{-1}d$. Denote $z=z_1\alpha_1 +\cdots+z_n\alpha_n$. We thus obtain the following system ($S$)

\small{
$$
 \left\{
\begin{array}{ll}
(n+1)z_1 =nd_1 + (n-1)d_2 +\cdots+ d_n \\
(n+1)z_2 =(n-1)d_1 + (2n-2)d_2 +\cdots+ 2d_n  \\
\text{~~} \text{~~} \text{~~} \text{~~} \text{~~} \text{~~} \text{~~} \text{~~} \text{~~} \text{~~} \text{~~} \text{~~} \text{~~} \text{~~}  \vdots \\
(n+1)z_k = [(n+1)\text{min}(k,1) -k]d_1 + [(n+1)\text{min}(k,2) -2k]d_2+\cdots+[(n+1)\text{min}(k,n) -nk]d_n \\
\text{~~} \text{~~} \text{~~} \text{~~} \text{~~} \text{~~} \text{~~} \text{~~} \text{~~} \text{~~} \text{~~} \text{~~} \text{~~} \text{~~} \vdots~\\
(n+1)z_n =d_1 + 2d_2 +\cdots+ nd_n.
\end{array}
\right. 
$$
}

Therefore, if there exists a $z$ satisfying this system, since all the $z_i$ belong to $\mathbb{Z}$, it must also satisfy for all $k=1,\dots ,n$ the following condition 
$$
\sum\limits_{j=1}^n\big[(n+1)\text{min}(k,j)-kj\big]d_j \in (n+1)\mathbb{Z},
$$
which is equivalent to the following equation in $\mathbb{Z}/(n+1)\mathbb{Z}$
$$
\sum\limits_{j=1}^n\big[(n+1)\text{min}(k,j)-kj\big]\overline{d_j} =0.
$$

We claim now that if $d_1 +2d_2 +\cdots+nd_n \in (n+1)\mathbb{Z}$ then we have $\sum\limits_{j=1}^n\big[(n+1)\text{min}(k,j)-kj\big]d_j \in (n+1)\mathbb{Z}$ for all other indices $k$. Since $d_1 +2d_2 +\cdots+nd_n \in (n+1)\mathbb{Z}$ there exists $r \in \mathbb{Z}$ such that $d_1 +2d_2 +\cdots+nd_n = (n+1)r$, that is $d_1 = (n+1)r-2d_2 -\cdots -nd_n$. Therefore 
\begin{align*}
 \sum\limits_{j=1}^n\big[(n+1)\text{min}(k,j)-kj\big]d_j  & =  (n+1-k)d_1 + \sum\limits_{j=2}^n\big[(n+1)\text{min}(k,j)-kj\big]d_j  \\
& = (n+1-k)((n+1)r-2d_2 -\cdots -nd_n)+ \sum\limits_{j=2}^n\big[(n+1)\text{min}(k,j)-kj\big]d_j \\
& = r(n+1-k)(n+1) + \sum\limits_{j=2}^n \big[(n+1)\text{min}(k,j)-kj-j(n+1-k)\big]d_j \\
& = r(n+1-k)(n+1) + \sum\limits_{j=2}^n \big[(n+1)(\text{min}(k,j)-j)\big]d_j \\
& = (n+1) \Big[ r(n+1-k) + \sum\limits_{j=2}^n \big[\text{min}(k,j)-j\big]d_j \Big].
\end{align*}

To summarize, if there exists $z \in \mathbb{Z}\Phi$ such that ($S$) is satisfied then 
we must have $${
\sum\limits_{j=1}^n\big[(n+1)\text{min}(k,j)-kj\big]\overline{d_j} =0
}$$ for all $k=1,\dots ,n$, which is equivalent to the equation $\overline{d_1} + 2\overline{d_2} +\cdots+ n\overline{d_n} = 0$. Therefore, if $\overline{d_1} + 2\overline{d_2} +\cdots+ n\overline{d_n} \neq 0$ there doesn't exist $z \in \mathbb{Z}\Phi$ such that $z={}^{t}C^{-1}d$, whence $\mathsf{s}_1$ and $\mathsf{s}_2$ define two different classes in $H^1(W,\mathbb{Z}\Phi)$. 

Conversely, if $\mathsf{s}_1$ and $\mathsf{s}_2$ define two different elements in $H^1(W,\mathbb{Z}\Phi)$ we do not have $z \in \mathbb{Z}\Phi$ satisfying ($S$). However, if $\overline{d_1} + 2\overline{d_2} +\cdots+ n\overline{d_n} = 0$ it follows that ${
\sum\limits_{j=1}^n\big[(n+1)\text{min}(k,j)-kj\big]\overline{d_j} =0
}$ for all $k=1,\dots ,n$. Thus, by setting $z_k := \frac{1}{n+1}\Big[\sum\limits_{j=1}^n\big[(n+1)\text{min}(k,j)-kj\big]d_j\Big]$ for all $k=1,\dots ,n$ we have built an element in $\mathbb{Z}\Phi $ such that $z={}^{t}C^{-1}d$, i.e. $\mathsf{s}_1$ and $\mathsf{s}_2$ are equivalent, which is impossible according to our assumption. 
\end{proof}

\begin{remark}\label{coro}
 Assume that $n+1$ is a prime number. Then for each pair of sections $\mathsf{s}_1$ and $\mathsf{s}_2$ that satisfy $d_i(\mathsf{s}_1,\mathsf{s}_2)=d_j(\mathsf{s}_1,\mathsf{s}_2)$ for all $i,j$, we have $\overline{\mathsf{s}}_1 = \overline{\mathsf{s}}_2$.
Indeed, let $\mathsf{s}_1$ and $\mathsf{s}_2$ be two sections satisfying $d_i(\mathsf{s}_1,\mathsf{s}_2)=d_j(\mathsf{s}_1,\mathsf{s}_2)$ for all $i,j$. 

Since $n+1$ is a prime number, the polynomial $X + 2X +\cdots+ nX$ admits all the elements of $\mathbb{Z}/(n+1)\mathbb{Z}$ as solutions. In particular, the equality $\overline{d_1} + 2\overline{d_2} +\cdots+ n\overline{d_n} =0$ is true. Therefore, because of Proposition \ref{H1 dans A} these two sections define the same element in $H^1(W(A_n),\mathbb{Z}\Phi)$.

Note that if $n+1$ is not prime, the above result does not necessarily hold. Indeed, assume that $n=3$ and let us take $\mathsf{s}_1$ that we identify with $(3,4,5)$ and $\mathsf{s}_2$ with $(6,7,8)$. We have here $d_i(\mathsf{s}_1,\mathsf{s}_2)= 3$ for $i=1,2,3$. Moreover $\overline{d_1}+2\overline{d_2}+3\overline{d_3} = \overline{2}\neq 0$. Hence $\overline{\mathsf{s}_1} \neq \overline{\mathsf{s}_2}$.
\end{remark}

\subsubsection{\texorpdfstring{Type $B_n$}{Type Bn}}\label{subB}

In this section $W=W(B_n)$ and $\Phi = B_n$. We denote by $C:=C_{B_n}$ the Cartan matrix of $B_n$. The coefficients of $C^{-1}$ are known (see \cite{YY}) and are given by 
\begin{equation}
(C^{-1})_{ij} = \frac{\text{min}(i,j)}{1-\text{min}(0,n-i-1)} =  \left\{
\begin{array}{ll}
\text{min}(i,j) &\text{if}~ i <n \\
\frac{j}{2} & \text{if}~ i=n
\end{array}
\right. 
1 \leq i,j \leq n.
\end{equation}

\begin{proposition}
The sections $\mathsf{s}_1$ and $\mathsf{s}_2$ define two different elements in $H^1(W, \mathbb{Z}\Phi)$ if and only if 
\begin{equation*}
    \overline{d_n} = 1 \quad \text{in} \quad \mathbb{Z}/2\mathbb{Z}.
\end{equation*}

\end{proposition}

\begin{proof}
By (\ref{cartan equation}) we know that $\mathsf{s}_1 \thicksim \mathsf{s}_2$ if and only if there exists $z \in \mathbb{Z}\alpha_1 \times \cdots \times \mathbb{Z}\alpha_n$ such that $z={}^{t}C^{-1}d$. Denote $z=z_1\alpha_1 +\cdots+z_n\alpha_n$. The system ($S$) associated to the equation $z={}^{t}C^{-1}d$ is as follows
$$
 \left\{
\begin{array}{ll}
2z_1 =2d_1 + 2d_2 +\cdots+ 2d_{n-1}+ d_n \\ 
~~~~~~~~~~~~~~\vdots \\
2z_k = 2\text{min}(k,1)d_1 + 2\text{min}(k,2)d_2+\cdots+2\text{min}(k,n-1)d_{n-1}+kd_n \\
~~~~~~~~~~~~~~\vdots~\\
2z_n =2d_1 + 2d_2 +\cdots+ 2(n-1)d_{n-1}+nd_n.
\end{array}
\right. 
$$

Therefore,  $\mathsf{s}_1 \thicksim \mathsf{s}_2$ if and only if $kd_n \in 2\mathbb{Z}$ for all $k=1,\dots,n$, that is if and only if $\overline{d_n} = 0$ in $\mathbb{Z}/2\mathbb{Z}$. This ends the proof.
\end{proof}

\bigskip

\subsubsection{\texorpdfstring{Type $C_n$}{Type Cn}}\label{subC} In this section $W=W(C_n)$ and $\Phi = C_n$. We denote by $C:=C_{C_n}$ the Cartan matrix of $C_n$. The coefficients of $C^{-1}$ are known (see \cite{YY}) and are given by 

\begin{equation}
(C^{-1})_{ij} = \frac{\text{min}(i,j)}{1-\text{min}(0,n-j-1)} =  \left\{
\begin{array}{ll}
\text{min}(i,j) &\text{if}~ j <n \\
\frac{i}{2} & \text{if}~ j=n
\end{array}
\right. 
1 \leq i,j \leq n.
\end{equation}

\begin{proposition}
 Write $I_n:=\{k \in \llbracket 1,n \rrbracket~|~k~ \text{is odd}\}$. The sections $\mathsf{s}_1$ and $\mathsf{s}_2$ define two different elements in $H^1(W, \mathbb{Z}\Phi)$ if and only if 
 \begin{equation*}
     \sum\limits_{k \in I_n}\overline{d_k} = 1 \quad \text{in} \quad \mathbb{Z}/2\mathbb{Z}.
 \end{equation*}

\end{proposition}

\begin{proof}
Denote $z=z_1\alpha_1 +\cdots+z_n\alpha_n$. Similarly as the $A_n$ and $B_n$ cases above, we obtain a system ($S$) associated to the equation $z={}^{t}C^{-1}d$, which is as follows
$$
 \left\{
\begin{array}{ll}
2z_1 =2d_1 + 2d_2 +\cdots+ 2d_n \\ 
~~~~~~~~~~~~~~\vdots \\
2z_k = 2\text{min}(k,1)d_1 + 2\text{min}(k,2)d_2+\cdots+2\text{min}(k,n)d_n \\
~~~~~~~~~~~~~~\vdots~\\
2z_n =d_1 + 2d_2 +\cdots+ nd_n.
\end{array}
\right. 
$$

It is obvious that $2\text{min}(k,1)d_1+\cdots+2\text{min}(k,n)d_n \in 2\mathbb{Z}$ for all  $k=1,\dots ,n-1$. Therefore, the only equation that matters in this system is the last one: $2z_n=d_1 + 2d_2 +\cdots+ nd_n \in 2\mathbb{Z}$. This equation transferred in $\mathbb{Z}/2\mathbb{Z}$ becomes $\sum\limits_{k \in I_n}k\overline{d_k} = 0$ and then $\sum\limits_{k \in I_n}\overline{d_k} = 0$. The result follows.
\end{proof}

\subsubsection{\texorpdfstring{Type $D_n$}{Type Dn}}\label{subD} In this section $W=W(D_n)$ and $\Phi = D_n$. We denote by $C:=C_{D_n}$ the Cartan matrix of $D_n$. It turns out that $C^{-1}$ is symmetric, this is why we just give the coefficients of $C^{-1}$ with entries  $1\leq i \leq j \leq n$

\begin{equation}
(C^{-1})_{ij} = \left\{
\begin{array}{ll}
i &\text{if}~ 1 \leq i \leq j \leq n-2 \\
\frac{i}{2} & \text{if}~ i<n-1, j=n-1 ~\text{or}~n\\
\frac{n-2}{4} & \text{if}~ i=n-1, j=n\\
\frac{n}{4} & \text{if}~ i=j=n-1 ~\text{or}~n.\\
\end{array}
\right. 
\end{equation}

A better visualization of this matrix is given by: 

\small{
$$
C^{-1}= \frac{1}{4}
  \begin{pmatrix}[ccccc|cc]
   4         & 4          & 4          & \cdots & 4           & 2                  & 2 \\
   4         & 8          & 8           & \cdots & 8          & 4                  & 4 \\
   4         & 8          & 12         & \cdots & 12         & 6                  & 6 \\
   \vdots & \vdots  & \vdots  & \ddots & \vdots  & \vdots          & \vdots \\
    4       & 8          & 12         & \cdots & {\scriptstyle 4(n-2) }    & {\scriptstyle2(n-2)}     & {\scriptstyle2(n-2)} \\ \hline
    2         & 4          & 6         & \cdots &  {\scriptstyle2(n-2)}       & {\scriptstyle n}   & {\scriptstyle n-2} \\
     2         & 4          & 6         & \cdots &  {\scriptstyle2(n-2)}       & {\scriptstyle n-2}            & {\scriptstyle n} \\
\end{pmatrix}.
$$
}

\bigskip

\begin{proposition}
Let $I_{n-2}:=\{k \in \llbracket 1,n-2 \rrbracket~|~k~ \text{is odd}\}$. The sections $\mathsf{s}_1$ and $\mathsf{s}_2$ define two different elements in $H^1(W, \mathbb{Z}\Phi)$ if and only if one of the following points is satisfied
\begin{enumerate}
\item[(i)] $\overline{d_{n-1}} \neq \overline{d_n}$ in $\mathbb{Z}/2\mathbb{Z}$,
\item[(ii)]  $\sum\limits_{k \in I_{n-2}}2\overline{d_k} + n\overline{d_{n-1}} + (n-2)\overline{d_n} \neq 0$ in $\mathbb{Z}/4\mathbb{Z}$,
\item[(iii)] $\sum\limits_{k \in I_{n-2}}2\overline{d_k} + (n-2)\overline{d_{n-1}} + n\overline{d_n} \neq 0$ in $\mathbb{Z}/4\mathbb{Z}$.
\end{enumerate}

\end{proposition}

\begin{proof}
By (\ref{cartan equation}) we know that $\mathsf{s}_1 \thicksim \mathsf{s}_2$ if and only if there exists $z \in \mathbb{Z}\alpha_1 \times \cdots \times \mathbb{Z}\alpha_n$ such that $z={}^{t}C^{-1}d$, that is if and only if $z=C^{-1}d$.
Denote $z=z_1\alpha_1 +\cdots+z_n\alpha_n$. The first line of the corresponding system ($S$) is given by:
$$4z_1 = 4d_1 + 4d_2 +\cdots+ 4d_{n-2} + 2d_{n-1} + 2d_n.$$

Thus, we must have $4d_1 + 4d_2 +\cdots+ 4d_{n-2} + 2d_{n-1} + 2d_n \in 4\mathbb{Z}$, that is $2d_{n-1} + 2d_n \in 4\mathbb{Z}$, which is equivalent to $\overline{d_{n-1}} = \overline{d_n}$ in $\mathbb{Z}/2\mathbb{Z}$. We see then that if $\overline{d_{n-1}} = \overline{d_n}$, the $(n-2)$ first equations of ($S$) are satisfied. 

The penultimate equation of ($S$) is
$$
4z_{n-1} = \sum\limits_{k=1}^{n-2}2kd_k + nd_{n-1} + (n-2)d_n.
$$
Therefore, this equality compels us to have 
$$
\sum\limits_{k=1}^{n-2}2kd_k + nd_{n-1} + (n-2)d_n \in 4\mathbb{Z},
$$ 
that is  ${\sum\limits_{k=1}^{n-2}2k\overline{d_k} + n\overline{d_{n-1}} + (n-2)\overline{d_n} = 0}$ in $\mathbb{Z}/4\mathbb{Z}$. If $k$ is even $2k\overline{d_k}=0$, and if $k$ is odd $2k\overline{d_k}=2\overline{d_k}$. Hence we have 
$$
\sum\limits_{k=1}^{n-2}2k\overline{d_k} + n\overline{d_{n-1}} + (n-2)\overline{d_n} = \sum\limits_{k \in I_{n-2}}2\overline{d_k} + n\overline{d_{n-1}} + (n-2)\overline{d_n} = 0.$$

The reasoning is exactly the same for the last equation. This concludes the proof.
\end{proof}

\section{Orientation of alcoves}\label{Sec orientation alcoves}

\subsection{Intuitive notion of orientation}\label{intuitive section}
Informally, the orientation of an alcove $A_w$ describes how its shape and configuration relate to the fundamental alcove $A_e$. We say that
two alcoves have the same orientation if they have the same shape and the same configuration.

We can be more precise about what we mean by the terms \say{shape} and \say{configuration}.
We illustrate this in the examples of types $A_2$ and $B_2$ below.

\subsubsection{\texorpdfstring{Type $A_2$}{Type A2}}
In the figure below, wall colors indicate the right weak order:
crossing a wall of a given color corresponds to applying the associated simple reflection on the right.
Pink corresponds to $s_0$, purple to $s_1$, and yellow to $s_2$.

\bigskip

\begin{figure}[h!]
    \centering
    \includegraphics[scale = 0.41]{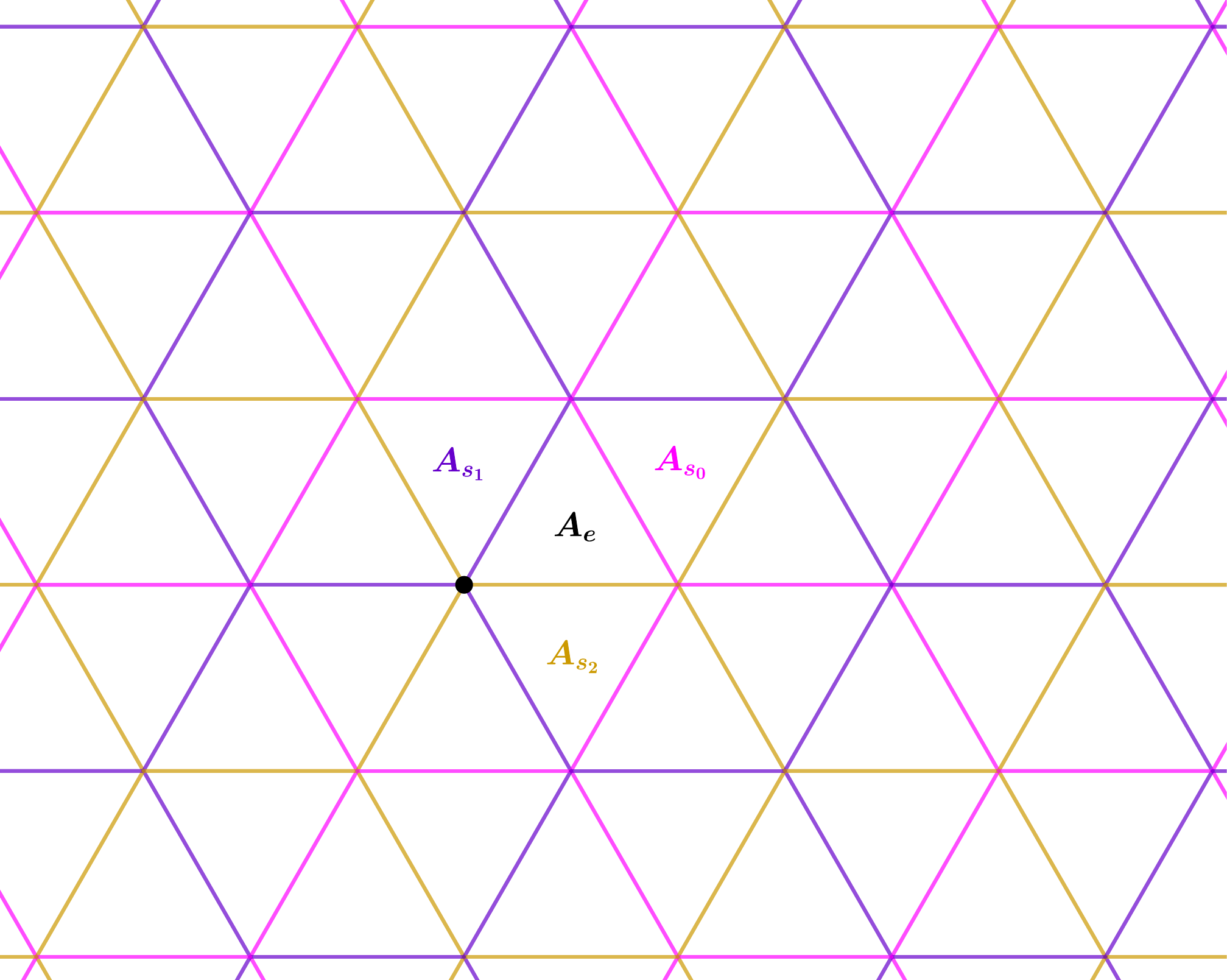}
    \caption{Alcoves in affine type $A_2$.}
    \label{affine_alcoves_A2}
\end{figure}

\bigskip

We have two different shapes, given by the elements of the fundamental parallelepiped:  $\text{Alc}(P_{A_2}) = \{e,\, s_0\}$ (see Section \ref{section parallelepiped}), and drawn below.

\bigskip

\begin{figure}[h!]
    \centering
    \includegraphics[scale = 0.34]{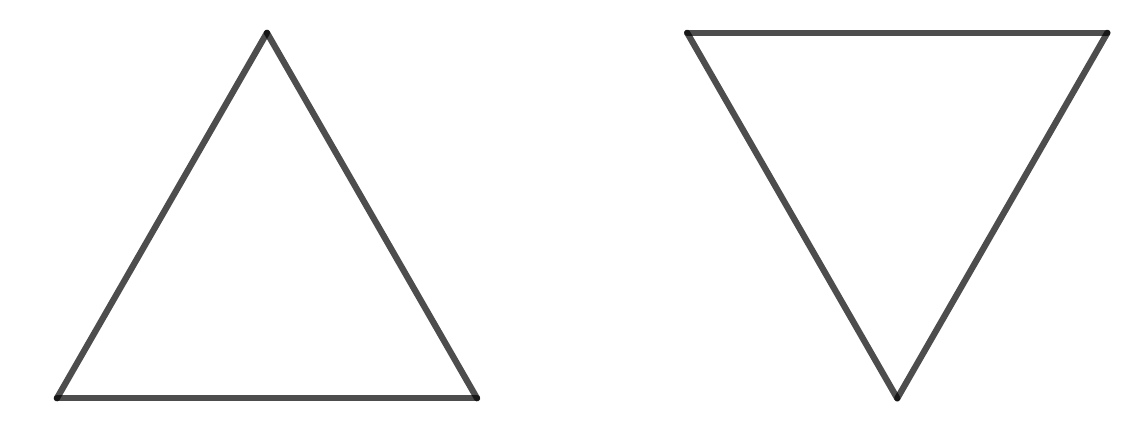}
    \caption{The two shapes in type $A_2$.}
    \label{the two shapes}
\end{figure}

\bigskip

But now, for each shape we have 3 different configurations, that is 3 different wall colorings.

\begin{figure}[h!]
    \centering
    \includegraphics[scale = 0.34]{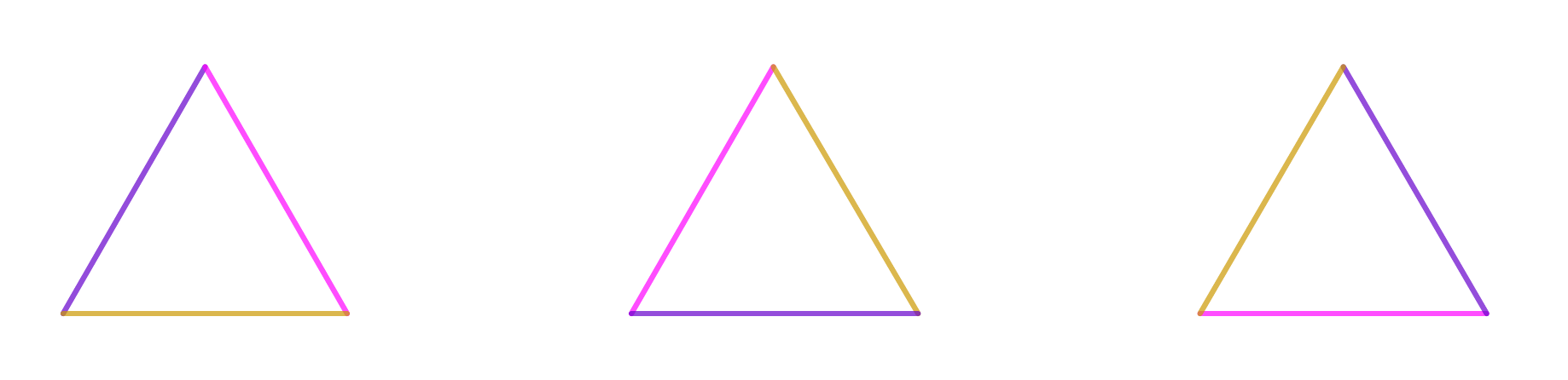}
    \caption{The 3 configurations of the first shape.}
    \label{config shape 1}
\end{figure}

\medskip

\begin{figure}[h!]
    \centering
    \includegraphics[scale = 0.34]{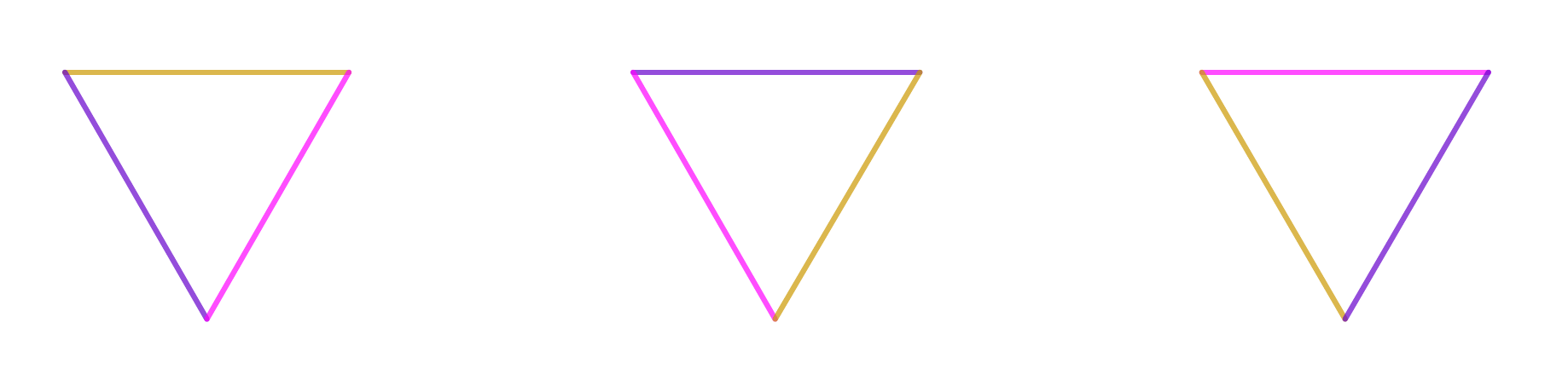}
    \caption{The 3 configurations of the second shape.}
    \label{config shape 2}
\end{figure}

\subsubsection{\texorpdfstring{Type $B_2$}{Type B2}}
Consider the alcoves of affine type \(B_2\). As in type $A_2$, the wall coloring (rather than the hyperplane coloring) encodes the right weak order.

The set of positive roots is
$
\Phi^+ = \{\alpha_1,\ \alpha_2,\ \alpha_1+\alpha_2,\ \alpha_1+2\alpha_2\}.
$
Thus there are four families of affine hyperplanes. Yet the associated affine Weyl group has rank $3$, so an alcove is bounded by only three hyperplanes. This mismatch introduces a subtlety that does not appear in type $A_2$.

\begin{figure}[h!]
    \centering
    \includegraphics[scale = 0.36]{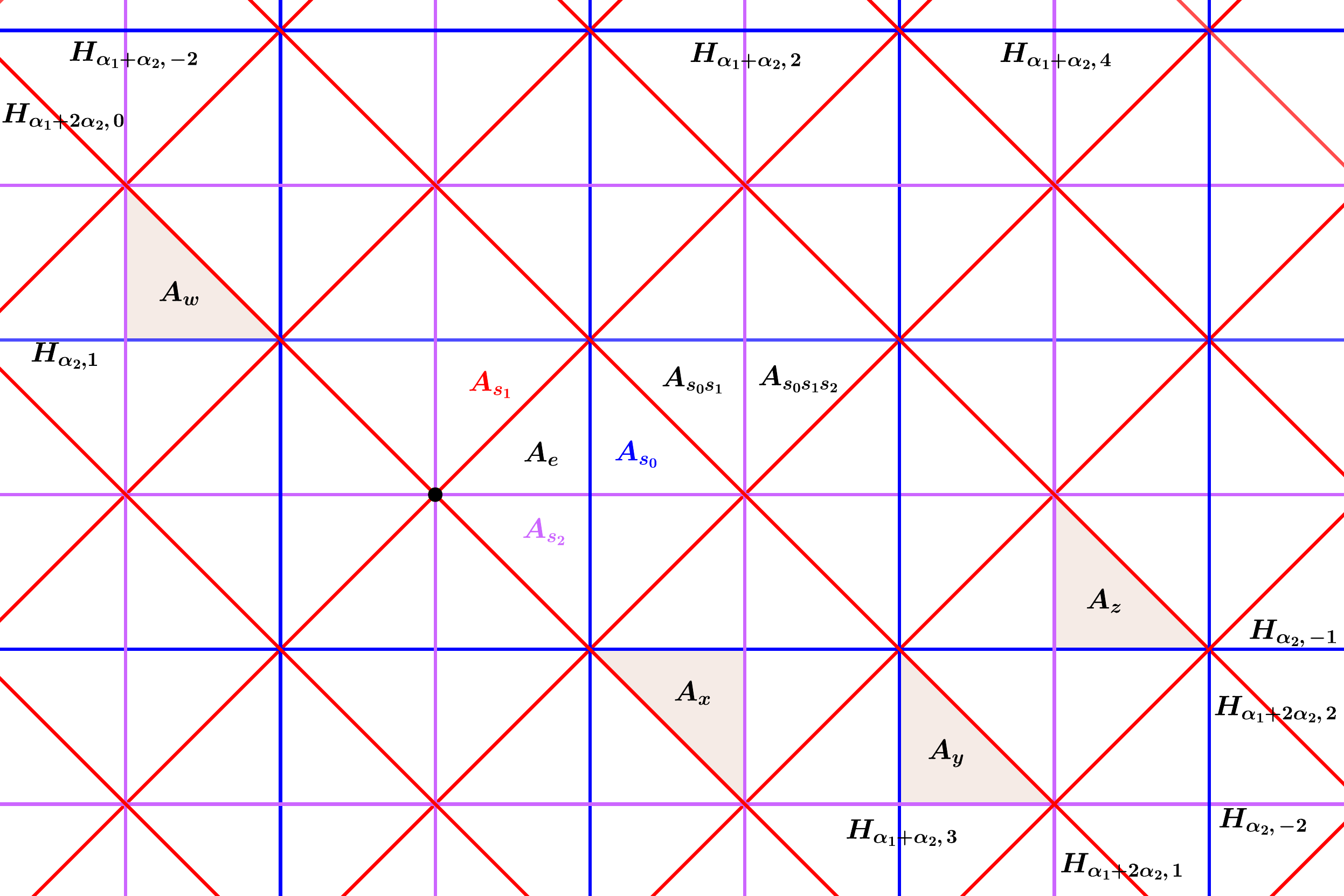}
    \caption{Alcoves in affine type $B_2$. The labeled hyperplanes are precisely those that bound the four alcoves \(A_w\), \(A_x\), \(A_y\), and \(A_z\). The four elements of $P_{B_2} = \{e,s_0, \, s_0s_1, \, s_0s_1s_2\}$ are also indicated.}
    \label{Alcoves B_2}
\end{figure}

The subtlety mentioned above is as follows. One might expect that having the same shape is equivalent
to having the same set of bounding hyperplanes. This is not true. For instance, the alcoves $A_w$
and $A_x$ share the same family of bounding hyperplanes
$$
\{H_{\alpha_2,*},\, H_{\alpha_1+\alpha_2,*},\, H_{\alpha_1+2\alpha_2,*}\},
$$
yet they do not have the same shape. Only one implication holds: having the same shape forces the
same family of bounding hyperplanes. This distinction may seem innocuous in low rank, but it becomes difficult to distinguish shapes as the rank grows in $n$, whereas $\Phi^+(B_n)$ grows in $n^2$. 
In fact, the number of possible shapes is known in general: it is equal to the number of alcoves
contained in $P_{\Delta}$; see~(\ref{number alcove funda para}).

Therefore in this case we have four different shapes corresponding to the elements of $\text{Alc}(P_{B_2}) = \{e,\, s_0, \,s_0s_1,\,s_0s_1s_2\}$.

\begin{figure}[h!]
    \centering
    \includegraphics[scale = 0.36]{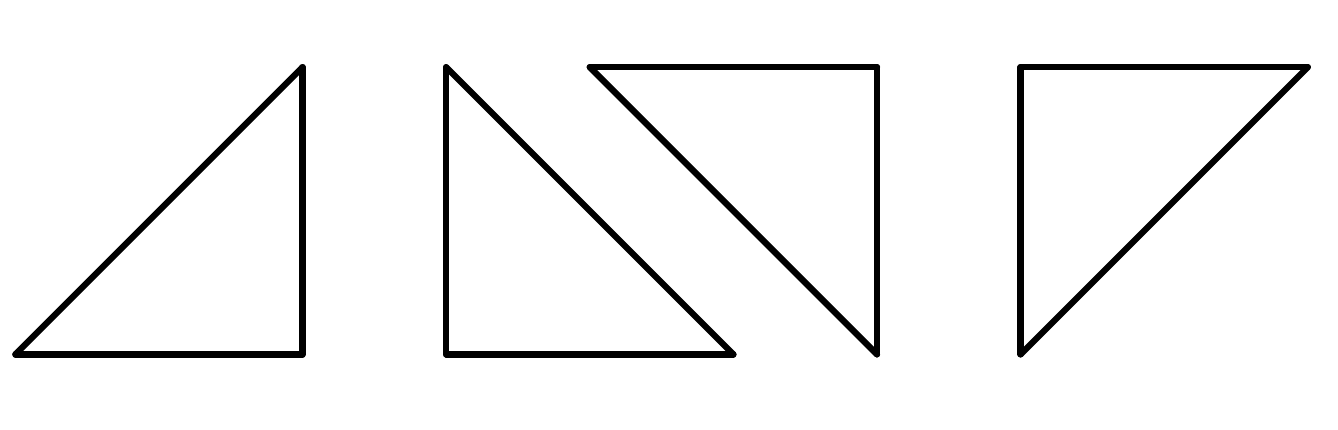}
    \caption{The four shapes in type $B_2$.}
    \label{shape Alcoves B_2}
\end{figure}

\medskip

For each shape we have two different configurations.

\begin{figure}[h!]
    \centering
    \includegraphics[scale = 0.36]{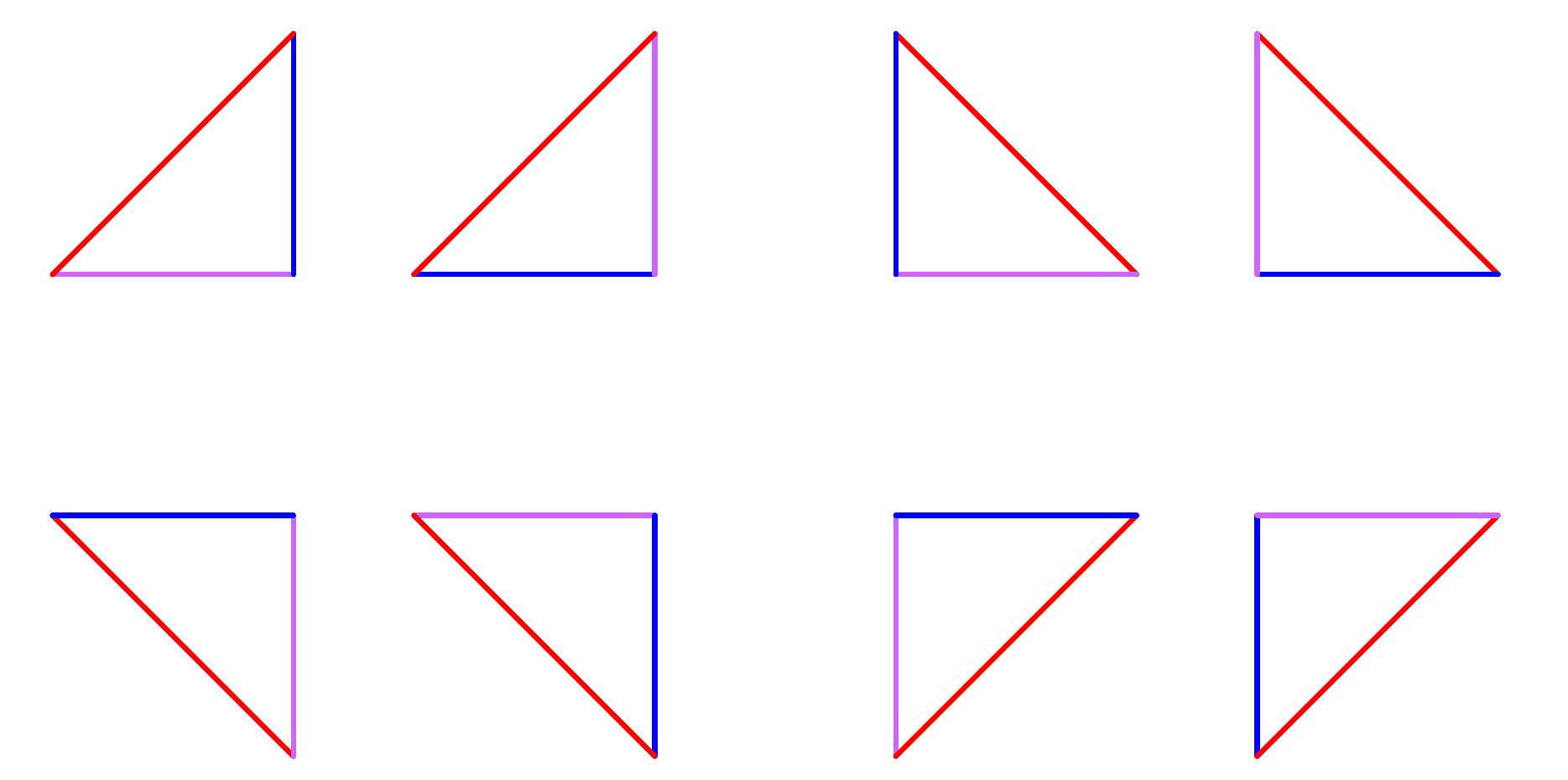}
    \caption{The two configurations of each shape in type $B_2$.}
    \label{orientatation Alcoves B_2}
\end{figure}

In Figure~\ref{Alcoves B_2}, the alcoves $A_y$ and $A_z$ share the same family of bounding hyperplanes
and the same shape; however, they have different configurations. Consequently, they do not have the
same orientation. By contrast, $A_w$ and $A_z$ have the same shape and the same configuration, and
therefore the same orientation.

\subsection{Formal definition of orientation and main results}

The formal definition of the orientation is the following: two alcoves $A_w$ and $A_{w'}$ have the same orientation if and only if $w'=\tau_xw$ for some $x \in \mathbb{Z}\Phi$. Since the irreducible components of $\widehat{X}_{W_a}$ are stable under the action of $\mathbb{Z}\Phi$, we can restrict the definition of having the same orientation to the components, namely $A_w$ and $A_{w'}$ can have the same orientation  only if $\iota(w)$ and $\iota(w')$ belong to the same component.

Let $w \in W_a$. We define $\mathsf{s}_w$ to be the section of (\ref{eq seq Z}) given by $\mathsf{s}_w(s_i) = \tau_{k(w,\alpha_i)\alpha_i}s_i$ for all $\alpha_i \in \Delta$. 

\begin{theorem}\label{coho eq orientation}
Let $w, w' \in W_a$ such that $\iota(w)$ and $\iota(w')$ are in the same component. Then $A_w$ and $A_{w'}$ have the same orientation if and only if $\mathsf{s}_{w}$ and $\mathsf{s}_{w'}$ define the same element in $H^1(W,\mathbb{Z}\Phi)$.
\end{theorem}

\begin{proof}
By \cite[Theorem 3.3]{JYS1}, we know that $k(\tau_xu, \alpha) = k(u,\alpha) + \langle x, \alpha^{\vee} \rangle$ for any $\alpha \in \Phi^+$ and any $u \in W$. This formula extends to any $u \in W_a$. Indeed for $u=\tau_y \overline{u}$ we have 
\begin{align*}
    k(\tau_xu, \alpha) & = k(\tau_x \tau_y \overline{u}, \alpha)  = k(\overline{u},\alpha) + \langle x+y, \alpha^{\vee} \rangle = k(\overline{u},\alpha) + \langle x, \alpha^{\vee} \rangle + \langle y, \alpha^{\vee} \rangle \\
                       & = k(u,\alpha) + \langle x, \alpha^{\vee} \rangle.
\end{align*}

 By definition, $A_w$ and $A_{w'}$ have the same orientation if and only if there exists $x \in \mathbb{Z}\Phi$ such that $w' = \tau_xw$. Since Shi coefficients characterize elements of $W_a$, $A_w$ and $A_{w'}$ have the same orientation if and only if there exists $x \in \mathbb{Z}\Phi$ such that $k(w',\alpha) = k(\tau_xw, \alpha)$ for any $\alpha \in \Phi^+$. By the above formula it follows that $A_w$ and $A_{w'}$ have the same orientation if and only if $k(w',\alpha)-k(w,\alpha) = \langle x, \alpha^{\vee} \rangle$ for any $\alpha \in \Phi^+$.  

 \medskip
 
  $\bullet$ Let us show the first direction. If $A_w$ and $A_{w'}$ have the same orientation then in particular $k(w',\alpha)-k(w,\alpha) = \langle x, \alpha^{\vee} \rangle$ for any $\alpha \in \Delta$. Therefore, by Proposition \ref{bij coh} we have $\mathsf{s}_{w} \thicksim \mathsf{s}_{w'}$, that is $\overline{\mathsf{s}_{w}} = \overline{\mathsf{s}_{w'}}$ in $H^1(W,\mathbb{Z}\Phi)$.

   \medskip
 
  $\bullet$ Conversely, if $\mathsf{s}_{w} \thicksim \mathsf{s}_{w'}$ then by Proposition \ref{bij coh} there exists $x \in \mathbb{Z}\Phi$ such that $k(w',\alpha)-k(w,\alpha) = \langle x, \alpha^{\vee} \rangle$ for any $\alpha\in \Delta$. By Theorem \ref{polynome} we know that for any $\beta \in \Phi^+$ we have $k(w,\beta) = P_{\beta}(w) + \lambda_{\beta}(w) $ and $ k(w',\beta) = P_{\beta}(w') + \lambda_{\beta}(w')$. Since $\iota(w)$ and $\iota(w')$ are in the same component it follows that $\lambda_{\beta}(w) = \lambda_{\beta}(w')$ for any $\beta \in \Phi^+$. In particular we have 
\begin{align*}
k(w',\beta) - k(w,\beta) &= P_{\beta}(w')-P_{\beta}(w) \\
&= P_{\beta}\Big(\{k(w',\alpha)\}_{\alpha \in \Delta}\Big) - P_{\beta}\Big(\{k(w,\alpha)\}_{\alpha \in \Delta}\Big) \\
& = P_{\beta}\Big(\{k(w',\alpha)-k(w,\alpha)\}_{\alpha \in \Delta}\Big) ~\hspace{1.1cm}(\text{by linearity of~} P_{\beta}) \\
& = P_{\beta}\Big(\{\langle x, \alpha^{\vee} \rangle\}_{\alpha \in \Delta}\Big) \\
& = \left\langle x, P_{\beta}\big(\{\alpha^{\vee}\}_{\alpha \in \Delta}\big)\right\rangle ~\hspace{3cm}(\text{by linearity of~} P_{\beta}) \\
& = \left\langle x, \beta^{\vee} \right\rangle. ~\hspace{4.7cm}(\text{by Theorem~} \ref{polynome})
\end{align*}

Therefore $k(w',\beta) = k(w,\beta) +  \langle x, \beta^{\vee} \rangle$ for any $\beta \in \Phi^+$, that is  $k(w',\beta) = k(\tau_xw,\beta)$ for any $\beta \in \Phi^+$. Thus $w' = \tau_xw$ with $x \in \mathbb{Z}\Phi$, which means that $A_{w'}$ and $A_w$ have the same orientation.
\end{proof}

In the following corollary, whose proof stems directly from Theorem \ref{coho eq orientation} and Section \ref{concrete}, the simple root $\alpha_i$ is in each case the $i$-th simple root with the same conventions as in Section \ref{concrete}. We also use the notation $\overline{k}(w,\alpha):= \overline{k(w,\alpha)}$ for any $\alpha \in \Phi^+$ and any $w \in W_a$. Recall our notation: $I_{r}=\{k \in \llbracket 1,r \rrbracket~|~k~ \text{is odd}\}$.

\begin{corollary}\label{coro orientation} Let $w, w' \in W_a$ such that $\iota(w)$ and $\iota(w')$ are in the same component. Then $A_w$ and $A_{w'}$ have the same orientation if and only if
\begin{enumerate}
    \item Type $A_n$
      $$
      \sum\limits_{j=1}^n j\overline{k}(w,\alpha_j)= \sum\limits_{j=1}^n j\overline{k}(w',\alpha_j) \quad \text{in} \quad  \mathbb{Z}/(n+1)\mathbb{Z}.
      $$

      \item Type $B_n$
      $$
        \overline{k}(w,\alpha_n) = \overline{k}(w', \alpha_n)  \quad \text{in} \quad  \mathbb{Z}/2\mathbb{Z}.    
      $$

      \item Type $C_n$
      $$
      \sum\limits_{j\in I_n} \overline{k}(w,\alpha_j) = \sum\limits_{j \in I_n} \overline{k}(w',\alpha_j) \quad \text{in} \quad \mathbb{Z}/2\mathbb{Z}.
      $$

      \item Type $D_n$: the following situations hold
      \begin{enumerate}
          \item[(i)] $\overline{k}(w,\alpha_{n-1}) + \overline{k}(w,\alpha_n) = \overline{k}(w', \alpha_{n-1}) + \overline{k}(w',\alpha_n)$ in $\mathbb{Z}/2\mathbb{Z}$,
          \item[(ii)] $\sum\limits_{i \in I_{n-2}}2\overline{k}(w,\alpha_i) + n\overline{k}(w,\alpha_{n-1}) + (n-2)\overline{k}(w, \alpha_{n}) = \sum\limits_{i \in I_{n-2}}2\overline{k}(w',\alpha_i) + n\overline{k}(w',\alpha_{n-1}) + (n-2)\overline{k}(w',\alpha_{n})$ in $\mathbb{Z}/4\mathbb{Z}$,
          \item[(iii)] $\sum\limits_{i \in I_{n-2}}2\overline{k}(w,\alpha_i) + (n-2)\overline{k}(w,\alpha_{n-1}) + n\overline{k}(w,\alpha_{n}) = \sum\limits_{i \in I_{n-2}}2\overline{k}(w',\alpha_i) + (n-2)\overline{k}(w',\alpha_{n-1}) + n\overline{k}(w',\alpha_{n})$ in $\mathbb{Z}/4\mathbb{Z}$.
      \end{enumerate}
\end{enumerate}
\end{corollary}

\bigskip

\begin{example}[Type $A$]
 Set $V=\mathbb{R}^{n+1}$ with canonical basis $\{e_1,\dots, e_{n+1}\}$. A way to describe the  roots  of $A_n$ is by
$
\Phi=\{\pm (e_i - e_j) ~| ~ 1\leq i < j \leq n+1 \}
$
with simple system 
$
\Delta = \{ \alpha_i:=e_i - e_{i+1} ~| ~ 1\leq i <  n +1\},
$
and  positive roots 
$
\Phi^+=\{e_i - e_j ~| ~ 1\leq i < j \leq n+1 \}.
$

A convenient way to write a Shi vector $v=(v_{ij})_{1\leq i < j \leq n+1}$ is by putting its coordinates in a pyramidal shape. For example:
\begin{figure}[h!]
\begin{center}
\begin{tikzpicture} 
\node at (0,0) {$v_{12}$} ;
\node at (1,0) {$v_{23}$} ;
\node at (2,0) {$v_{34}$} ;
\node at (3,0) {$v_{45}$} ;
\node at (4,0) {$v_{56}$} ;
\node at (0.5,0.5 ) {$v_{13}$} ;
\node at (1.5, 0.5 ) {$v_{24}$} ;
\node at (2.5, 0.5) {$v_{35}$} ;
\node at (3.5, 0.5) {$v_{46}$} ;
\node at (1,1) {$v_{14}$} ;
\node at (2,1) {$v_{25}$} ;
\node at (3,1) {$v_{36}$} ;
\node at (1.5,1.5 ) {$v_{15}$} ;
\node at (2.5,1.5 ) {$v_{26}$} ;
\node at (2,2) {$v_{16}$} ;
\end{tikzpicture}
\end{center}
\caption{Presentation of the coordinates of a Shi vector for $n=5$ where $v_{ij}$ is the coordinate over the position $e_i - e_j \in \Phi^+$. With this presentation the coefficients over the simple roots are on the first line of the pyramid.}
\label{root A}
\end{figure}

By (\ref{Shi ineq A}) we have  $v_{ij} = v_{ik} + v_{kj} + \delta_{ikj}(v)$ where $\delta_{ikj} \in \{0,1\}$. Another way to see the relation \say{being in the same component of the Shi variety} is via the coefficients $\delta_{ikj}$. Indeed, two Shi vectors $v, v'$ are in the same component if and only if $\delta_{ikj}(v) = \delta_{ikj}(v')$  for all $1 \leq i <k < j \leq n+1$.

\begin{figure}[h!]
\begin{center}
\begin{tikzpicture} 
\node at (0,0) {$1$} ;
\node at (1,0) {$-2$} ;
\node at (2,0) {$0$} ;
\node at (3,0) {$3$} ;
\node at (4,0) {$7$} ;

\node at (0.5,0.5 ) {$\textcolor{red}{0}$} ;
\node at (1.5, 0.5 ) {$-2$} ;
\node at (2.5, 0.5) {$\textcolor{red}{4}$} ;
\node at (3.5, 0.5) {$\textcolor{red}{11}$} ;

\node at (1,1) {$\textcolor{red}{0}$} ;
\node at (2,1) {$\textcolor{red}{2}$} ;
\node at (3,1) {$11$} ;

\node at (1.5,1.5 ) {$\textcolor{red}{4}$} ;
\node at (2.5,1.5 ) {$9$} ;

\node at (2,2) {$\textcolor{red}{11}$} ;

\node at (6,0) {$0$} ;
\node at (7,0) {$-1$} ;
\node at (8,0) {$2$} ;
\node at (9,0) {$4$} ;
\node at (10,0) {$-3$} ;

\node at (6.5,0.5 ) {$\textcolor{red}{0}$} ;
\node at (7.5, 0.5 ) {$1$} ;
\node at (8.5, 0.5) {$\textcolor{red}{7}$} ;
\node at (9.5, 0.5) {$\textcolor{red}{2}$} ;

\node at (7,1) {$\textcolor{red}{2}$} ;
\node at (8,1) {$\textcolor{red}{6}$} ;
\node at (9,1) {$4$} ;

\node at (7.5,1.5 ) {$\textcolor{red}{7}$} ;
\node at (8.5,1.5 ) {$3$} ;

\node at (8,2) {$\textcolor{red}{4}$} ;
\end{tikzpicture}
\end{center}
\caption{Two Shi vectors in the same component. The red indicates the positions where we need to use $+1$ for at least one $k$.}
\end{figure}

Since these two Shi vectors are in the same component (because the same red pattern  appears), we can now ask whether they have the same orientation or not. By Corollary \ref{coro orientation} we see that the corresponding alcoves do not have the same orientation, since in $\mathbb{Z}/6\mathbb{Z}$ one has
$$
\overline{1} + 2\cdot\overline{-2}+3\cdot\overline{0} + 4\cdot\overline{3} + 5\cdot\overline{7} = \overline{2},
$$
and 
$$
\overline{0} + 2\cdot\overline{-1}+3\cdot\overline{2} + 4\cdot\overline{4} + 5\cdot\overline{-3} = \overline{5}.
$$

\end{example}

\bigskip

\begin{example}[Type $B$]
Let $\{e_1, e_2\}$ be the canonical basis of $\mathbb{R}^2$. In this case, a way to describe the positive roots is by
$
\Phi^+=\{e_1-e_2, \, e_2,\, e_1, \, e_1+e_2\},
$
with simple system 
$
\Delta = \{\alpha_1:=e_1-e_2, \, \alpha_2:=e_2\}.
$

 By Corollary \ref{coro orientation}, two alcoves $A_w$, $A_{w'}$ of $W(\widetilde{B}_2)$ belonging to the same component have the same orientation if and only if $\overline{k}(w,e_2) = \overline{k}(w',e_2)$ in $\mathbb{Z}/2\mathbb{Z}$, which can be checked on Figure \ref{orbits B_2}, with the conventions of Figure \ref{convention}.

\begin{figure}[h!]
\begin{center}
\begin{tikzpicture}

\node at (5,0) {$k(w,e_1-e_2)$} ;
\node at (7,0) {$k(w,e_2)$} ;
\node at (6,1.2) {$k(w,e_1)$} ;
\node at (6,2.4) {$k(w,e_1+e_2)$} ;

\draw[line width=0.3mm](5,0.2)--(6,1);
\draw[line width=0.3mm](7,0.2)--(6,1);
\draw[line width=0.3mm](6,1.4)--(6,2.2);

\node at (8.2,1.2) {$=$} ;

\node at (10,1.2) {$k(w,e_1)$} ;
\node at (12.8,1.2) {$k(w,e_2)$} ;
\node at (11.4,2.4) {$k(w,e_1-e_2)$} ;
\node at (11.4,0) {$k(w,e_1+e_2)$} ;

\end{tikzpicture}
\end{center}
\caption{On the left hand side is the presentation of the Shi vectors of $W(\widetilde{B_2})$ ordered by height as in Figure \ref{root A}, while on the right hand side is the way we put the Shi coefficients inside each alcove in Figure \ref{orbits B_2}.}
\label{convention}
\end{figure}
 In type $B_2$ there are four admitted vectors (corresponding to the four elements of $P_{B_2}$):
 $$
 \lambda_1 = (0,0,0,0) \quad \lambda_2 = (0,0,1,0) \quad \lambda_3 = (0,0,1,1) \quad \lambda_4 = (0,0,2,1)
 $$
 where the positions are given as follows 
 $$(\alpha_1,\, \alpha_2, \, \alpha_1 + \alpha_2, \, \alpha_1 +2\alpha_2) = (e_1 - e_2, \, e_2, \, e_1, \,e_1 + e_2).
 $$
 Therefore,  the Shi variety associated to $W(\widetilde{B}_2)$ is given by
\begin{equation*}
\widehat{X}_{\widetilde{B}_2} =  X_{\widetilde{B}_2}[\lambda_1] ~\sqcup ~X_{\widetilde{B}_2}[\lambda_2] \sqcup X_{\widetilde{B}_2}[\lambda_3] ~\sqcup~ X_{\widetilde{B}_2}[\lambda_4].
\end{equation*}

The integral points of each component correspond to the alcoves having the same color. The first component corresponds to the pink region, the second to the yellow, the third to the blue, and the last to the white. 

\medskip

In Figure \ref{orbits B_2}, the orientation becomes particularly easy to obtain in terms of Shi coefficients, and it agrees with the \say{intuitive} notion of orientation explained in Section~\ref{intuitive section}. 

Two alcoves $A_w$ and $A_{w'}$ have the same orientation if and only if the two following points are satisfied:
\begin{itemize}
    \item[(i)] They have the same color (meaning that they are in the same component),
    \item[(ii)] their coefficients $k(w, e_2)$ and $k(w',e_2)$ are equal modulo $2$.
\end{itemize}

\begin{figure}[h!]
\centering
\includegraphics[scale=0.23]{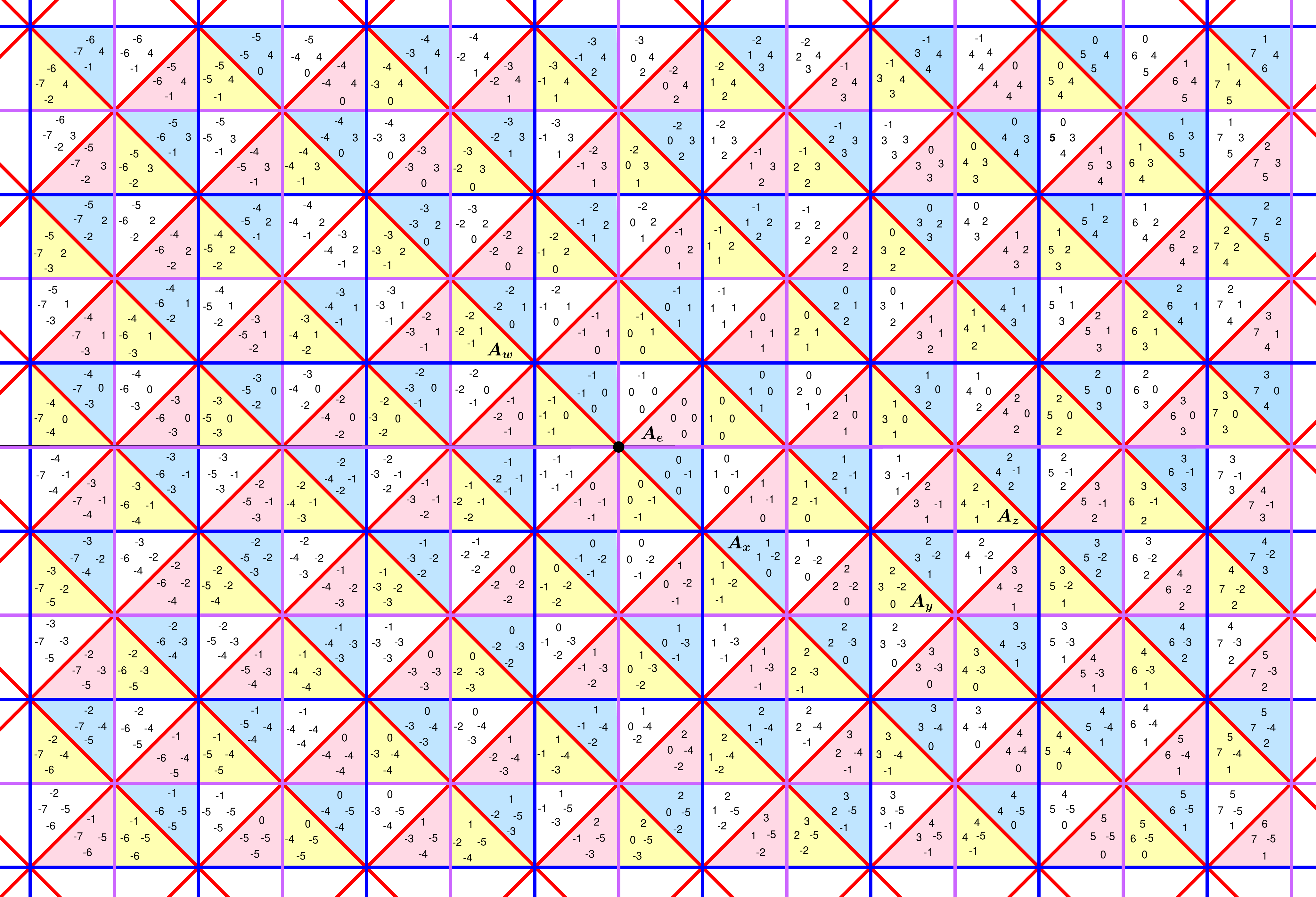} 
\caption{Alcoves in affine type $B_2$ with their Shi coefficients and their orientations. Each color corresponds to a component of the Shi variety. The alcoves $A_w$, $A_x$, $A_y$ and $A_z$ are the same as in Figure \ref{Alcoves B_2}.}
\label{orbits B_2}
\end{figure}
\end{example}

\newpage

\noindent \textbf{Acknowledgements}.
We thank Matthew Dyer and Hugh Thomas for providing helpful comments. We are grateful to the referees for their careful reading, which helped improve the paper. 
This work was supported by NSERC grants, by the LACIM and by the CNRS.

\bibliographystyle{plain}
\bibliography{Orientation}
\end{document}